\documentclass[a4paper,10pt]{article}
\usepackage[top=2.54cm,bottom=2.0cm,left=2.0cm,right=2.54cm, includeheadfoot]{geometry}

\usepackage{ulem}

\usepackage[T1]{fontenc}
\usepackage[utf8]{inputenc}
\usepackage[english]{babel}
\usepackage{enumerate}
\usepackage{multirow,booktabs}
\usepackage[table]{xcolor}
\usepackage{fullpage}
\usepackage{lastpage}
\usepackage{indentfirst}

\usepackage{footnote}

\usepackage{amsmath,amsfonts,amssymb,amscd,amsthm}

\usepackage{bm}

\usepackage[all,2cell]{xy} \UseAllTwocells \SilentMatrices

\usepackage{hyperref}

\usepackage{subfiles}

\usepackage{multicol}

\newtheorem{teo}{Theorem}[section]
\newtheorem{lem}[teo]{Lemma} 
\newtheorem{cor}[teo]{Corollary}
\newtheorem{prop}[teo]{Proposition} 

\newtheorem{defn}[teo]{Definition} 
\newtheorem{ex}[teo]{Example}

\newtheorem{fat}[teo]{Fact}
\newtheorem*{claim*}{Claim}

\newtheorem{rem}[teo]{Remark}

\begin{document}

\title{The general Arason-Pfister Hauptsatz}

\author{Kaique Matias de Andrade Roberto  \& Hugo Rafael de Oliveira Ribeiro \& Hugo Luiz Mariano\footnote{The authors want to express their gratitude to
Coordena\c c \~ao de Aperfei\c coamento de Pessoal de N\' ivel Superior Capes) -Brazil by the financial support to develop this work. The third author: program Capes-Print  number 88887.694866/2022-00.}}

\date{}

\maketitle


\begin{abstract}
    In the present we develop a fragment of the theory of superfields, polynomials and Marshall's quotient in order to obtain for general special groups, a proof of the Arason-Pfister Hauptsatz (APH): "if $\phi \neq \emptyset$ is an anisotropic form and $\phi \in I^n(F)$ then $dim (\phi) \geq 2^n$". In the process, we also obtain an alternative proof of APH for reduced special groups that avoid the uses of the invariants developed in \cite{dickmann2000special}. The applications of the full Arason-Pfister  Hauptsatz leads to interesting properties of graded rings associated to special groups/hyperfields.

\end{abstract}

\section{Introduction}

One of the more emblematic questions/results in algebraic and abstract theories of quadratic forms is the so-called Arason Pfister Hauptsatz (APH), which we will present a brief historic in the sequel. 

In \cite{milnor1970algebraick}, a 1970 paper of John Milnor  seminal to the algebraic theory of quadratic forms  over fields, the author poses two questions concerning the class of fields of characteristic $\neq 2$ (positively solved in the paper in many instances). One of the question was concerning the so called "Milnor's conjectures for the graded cohomology ring and for the graded Witt ring" that Voevodsky et al. solved around 2000. The other question asked if for every such field $F$, the intersection $\bigcap_{n \in \mathbb{N}} I^n(F)$ contains only $0 \in W(F)$, where $ I^n(F)$ is the n-th power of the fundamental ideal $I(F)$ of the Witt ring of $F$ ($I(F) = \{$even dimensional anisotropic forms over $F\}$). 

In the subsequent year, J. Arason and A. Pfister solved this question as an immediate corollary of the nowadays called "Arason-Pfister Hauptsatz" (APH):

(\cite{arason1971hauptsatz})\textbf{Let $\phi \neq \emptyset$ be an anisotropic form. If $\phi \in I^n(F)$, then $dim (\phi) \geq 2^n$.}

The theory of special groups, an abstract (first-order) theory of quadratic forms developed by Dickmann-Miraglia since the middle of the 1990s, allows a functorial encoding of the algebraic theory of quadratic forms of fields (with char $\neq 2$). In \cite{dickmann2000special},  Dickmann-Miraglia, restated the APH to the setting of special groups and, employing boolean theoretic methods to define and calculate the Stiefel-Whitney and the Horn-Tarski invariants of a special group, establish a generalization of the APH to the setting of {\bf reduced} special groups, in particular proving a different proof of the APH for formally real Pythagorean fields.

The difficult to attack APH for general special groups consists in the fact that the methods available for reduced special groups (the invariants) and for special groups provenient from fields (quadratic and transcendent extensions, valuations and so on) have not a clear generalization for the general case.

In the present work we avoid these difficulties considering multivalued structures. With the machinery of multirings (developed by M. Marshall in \cite{marshall2006real}) and superrings (here we adopt the perspective in \cite{ameri2019superring}) we prove the content of Theorem \ref{haup}:

\begin{teo}[Arason-Pfister Hauptsatz]
 Let $F$ be a special hyperfield, then  it holds $AP_F(n)$, for all $n \geq 0$. In more details: for each  $n \geq 0$ and For each $\varphi = \langle a_1,\cdots, a_k \rangle$, a  non-empty ($k\geq 1$), regular ($a_i \in \dot{F}$) and anisotropic form, if  $\varphi\in I^n(F)$, then $\dim(\varphi)\ge2^n$  $\varphi\in I^n(F) $, if $\varphi \neq \emptyset$ is anisotropic, then $\dim_{W,F}(\varphi)\ge2^n$.
\end{teo}


The concept of multivalued structure --  ``algebraic like'' structures but endowed with  multiple valued operations -- has been studied since the 1930's (\cite{marty1934generalization}); in particular, the concept of hyperrings was introduced by Krasner in the 1950's(\cite{krasner1956approximation}). Some general algebraic study has been made on multialgebras: see for instance \cite{golzio2018brief} and \cite{pelea2006multialgebras}. The use of hyperfields/hyperrings/multirings in connection  with Real Algebraic Geometry started 15 years ago, in \cite{marshall2006real}. 

Since the middle of the 2000s decade, the notion of multiring have obtained more attention: a multiring is an hyperring, satisfying an weak distributive law, but  hyperfields and multifields coincide. Multirings   has been studied for  applications many areas: in real algebraic geometry and abstract quadratic forms theory (\cite{marshall2006real}, \cite{worytkiewiczwitt2020witt}, \cite{roberto2021quadratic}), tropical geometry (\cite{viro2010hyperfields}, \cite{jun2015algebraic}), algebraic geometry ((\cite{jun2021geometry}, \cite{baker2021descartes}), valuation theory (\cite{jun2018valuations}), Hopf algebras (\cite{eppolito2020hopf}), etc (\cite{baker2021structure}, \cite{ameri2020advanced}, \cite{ameri2017multiplicative}, \cite{bowler2021classification}).

A more detailed account of variants of concept of polynomials over hyperrings is even more recent, having less than five years (\cite{jun2015algebraic}, \cite{ameri2019superring}, \cite{baker2021descartes},  \cite{roberto2021superrings}). 

In \cite{roberto2021superrings} we start a model-theoretic oriented analysis of multialgebras  introducing the class of algebraically closed hyperfields and providing variant proof of quantifier elimination flavor,   based on new results on  superring of polynomials. In \cite{roberto2022ACmultifields2} we develop provide new steps the program of studying the hyperfields (and natural variants: superfields) under a natural notion of algebraic extension and roots of polynomials - this shares some common features with the recent work in \cite{baker2021descartes} - we showed that every superfield has a (unique up to isomorphism) full algebraic extension to a superfield that is algebraically closed.

In the present work we develop provide new steps the program of studying the special groups/hyperfields under a construction that emulate the natural notion of algebraic extension and roots of polynomials:
\begin{defn}[Definition \ref{sgalgext}]
    Let $G$ be a special group, $F$ its special hyperfield associated and $\alpha\in F$ with $\alpha\ne0,1$. We denote $\omega=[X]\in F[X]/\langle X^2-\alpha\rangle$ and define
    \begin{align*}
        F(\omega)&:=F[X]/\langle X^2-\alpha\rangle;\\
        S_F(\omega)&=(F(\omega))/_m(F(\omega)^2\setminus\{0\});\\
        S^{red}_F(\omega)&=(F(\omega))/_m\left(\sum F(\omega)^2\setminus\{0\}\right).
    \end{align*}
    We also write $\omega=\sqrt\alpha$.
\end{defn}
This "quadratic extension" $F(\omega)$ of a special hyperfield $F$ provides the main idea of an alternative proof of Hauptsatz for reduced special groups:
\begin{teo}[APH for reduced special groups \ref{haupred}]\label{haupintro}
 Let $G$ be a reduced special group, then  it holds $AP_G(n)$, for all $n \geq 0$. In more details: for each  $n \geq 0$ and For each $\varphi = \langle a_1,\cdots, a_k \rangle$, a  non-empty ($k\geq 1$), regular and anisotropic form, if $\varphi\in I^n(G)$, then $\dim(\varphi)\ge2^n$  $\varphi\in I^n(G)$, if $\varphi \neq \emptyset$ is anisotropic, then $\dim_{W,G}(\varphi)\ge2^n$.
\end{teo}
In our alternative proof of \ref{haupred} we mainly use the Marshall quotient of a hyperfield and avoid the use of the invariants developed in Chapter 7 of \cite{dickmann2000special}. These methods are available for general special groups, and the main Theorems of Section \ref{newresg} are fully generalized in Section \ref{Hauptsatz-section}. This is the machinery necessary for the proof of Theorem \ref{haupintro}.

\textbf{Outline of the work:}

For the benefit of the reader, in Section \ref{preliminaries-section} we provide  the main definitions and results on the special groups, multivalued structures and its relations which were considered in this work.

In Section \ref{newresg} we provide an alternative prof of APH for reduced special groups. In Section \ref{polynomial-section} we develop the background theory for Marshall quotient of superfields necessary for the attack of APH for general special groups.

Throughout a combination of the tools developed in Sections \ref{newresg} and \ref{polynomial-section}, we develop in Section \ref{Hauptsatz-section} a theory of a structure for special hyperfields that mimics quadratic extensions of fields. Theorems \ref{teo150}, \ref{teo32} and \ref{iso2} are the main results achieved. With these, we obtain an extension of the Arason-Pfister Hauptsatz to all special hyperfields (Theorem \ref{haupintro}). As an application of this extended version of  Arason-Pfister Hauptsatz we describe interesting properties of  graded ring Witt ring (\cite{dickmann1998quadratic}, \cite{dickmann2000special}) and graded k-theory ring associated to special hyperfields (\cite{dickmann2006algebraic}, \cite{roberto2021ktheory}): see Propositions \ref{MC}, \ref{k-stable} and \ref{epiK}.

\section{Summary on Special Groups, Hyperfields and Superfields}\label{preliminaries-section}

We provide here, for the benefit of the reader, the main definitions and results on the special groups, multivalued structures and its relations which were considered in this work.

\subsection{Multirings, Hyperfields}


\begin{defn}[Adapted from definition 1.1 in \cite{marshall2006real}]\label{defn:multimonoid}
 An \textbf{abelian or commutative multigroup} is a structure  $(G,\cdot,r,1)$ where $G$ is a non-empty set, 
$r:G\rightarrow G$ is a function, $1$ is an element of $G$, $\cdot:G\times G\rightarrow\mathcal P(G)$ is a binary multioperation (we denote $d\in a\cdot b$ for $d\in\cdot(a,b)$) such that for all 
$a,b,c,d\in G$:
 \begin{description}
 \item [M1 - ] If $c\in a\cdot b$ then $a\in c\cdot(r(b))$ and $b\in(r(a))\cdot c$. We write $a\cdot b^{-1}$ to simplify  $a\cdot(r(b))$.
 \item [M2 - ] $b\in a\cdot1$ iff $a=b$.
 \item [M3 - ] With the convention $a\ast(b\ast c)=\bigcup\limits_{w\in b\ast c}x\ast w$ and 
  $(a\ast b)\ast c=\bigcup\limits_{t\in a\ast b}t\ast c$,
  $$a\ast(b\ast c)=(a\ast b)\ast c.$$
 \item [M4 - ] $c\in a\cdot b$ iff $c\in b\cdot a$.
\end{description}
The structure $(G,\cdot,1)$ is a \textbf{commutative multimonoid (with unity)} if satisfy M3 and M4 and the condition 
$a\in1\cdot a$ for all $a\in G$.
\end{defn}

\begin{defn}[Adapted from Definition 2.1 in \cite{marshall2006real}]\label{defn:multiring}
 A multiring is a sextuple $(R,+,\cdot,-,0,1)$ where $R$ is a non-empty set, $+:R\times 
R\rightarrow\mathcal P(R)\setminus\{\emptyset\}$,
 $\cdot:R\times R\rightarrow R$
 and $-:R\rightarrow R$ are functions, $0$ and $1$ are elements of $R$ satisfying:
 \begin{enumerate}[i -]
  \item $(R,+,-,0)$ is a commutative multigroup;
  \item $(R,\cdot,1)$ is a commutative monoid;
  \item $a.0=0$ for all $a\in R$;
  \item If $c\in a+b$, then $c.d\in a.d+b.d$. Or equivalently, $(a+b).d\subseteq a.d+b.d$.
 \end{enumerate}

Note that if $a \in R$, then $0 = 0.a \in (1+ (-1)).a \subseteq 1.a + (-1).a$, thus $(-1). a = -a$.
 
 $R$ is said to be an hyperring if for $a,b,c \in R$, $a(b+c) = ab + ac$. 
 
 A multiring (respectively, a hyperring) $R$ is said to be a multidomain (hyperdomain) if it hasn't zero divisors. A multiring 
$R$ will be a 
multifield if every non-zero element of $R$ has 
multiplicative inverse; \emph{note that hyperfields and multifields coincide}.
\end{defn}

\begin{ex}\label{ex:1.3}
$ $
 \begin{enumerate}[a -]
  \item Suppose that $(G,\cdot,1)$ is a group. Defining $a \ast b = \{a \cdot b\}$ and $r(g)=g^{-1}$, 
we have that $(G,\ast,r,1)$ is a multigroup. In this way, every ring, domain and field is a multiring, 
multidomain and multifield, respectively.
  
  \item Let $K=\{0,1\}$ with the usual product and the sum defined by relations $x+0=0+x=x$, $x\in K$ and 
$1+1=\{0,1\}$. This is a hyperfield  called Krasner's hyperfield \cite{jun2015algebraic}. The prime ideals  of a commutative ring (its Zariski spectrum) are classified by equivalence classes of morphisms into  algebraically closed  fields,  but they can be {\em uniformly classified} by a multiring morphism into the hyperfield $K$.
  
  \item $Q_2=\{-1,0,1\}$ is hyperfield with the usual product (in $\mathbb Z$) and the multivalued sum defined by 
relations
  $$\begin{cases}
     0+x=x+0=x,\,\mbox{for every }x\in Q_2 \\
     1+1=1,\,(-1)+(-1)=-1 \\
     1+(-1)=(-1)+1=\{-1,0,1\}
    \end{cases}
  $$
  The orderings of a commutative ring (its real spectrum) are classified by classes of equivalence of ring homomorphims into  real closed fields, but they can be {\em uniformly classified} by a multiring morphism into the  hyperfield  $Q_2$. 
  \end{enumerate}
\end{ex}

In the sequence, we provide examples that generalizes the previous ones.

\begin{ex}[H-hyperfield, Example 2.8 in \cite{ribeiro2016functorial}]\label{H-multi}
Let $p\ge1$ be a prime integer and $H_p:=\{0,1,...,p-1\} \subseteq \mathbb{N}$. Now, define the binary multioperation and 
operation in $H_p$ as 
follow:
\begin{align*}
 a+b&=
 \begin{cases}H_p\mbox{ if }a=b,\,a,b\ne0 \\ \{a,b\} \mbox{ if }a\ne b,\,a,b\ne0 \\ \{a\} \mbox{ if }b=0 \\ \{b\}\mbox{ if 
}a=0 \end{cases} \\
 a\cdot b&=k\mbox{ where }0\le k<p\mbox{ and }k\equiv ab\mbox{ mod p}.
\end{align*}
$(H_p,+,\cdot,-, 0,1)$ is a hyperfield such that for all $a\in H_p$, $-a=a$. In fact, these $H_p$ is a kind of generalization of $K$, in the sense that $H_2=K$.
\end{ex}

\begin{ex}[Kaleidoscope, Example 2.7 in \cite{ribeiro2016functorial}]\label{kaleid}
 Let $n\in\mathbb{N}$ and define 
 $$X_n=\{-n,...,0,...,n\} \subseteq \mathbb{Z}.$$ 
 We define the \textbf{$n$-kaleidoscope multiring} by 
$(X_n,+,\cdot,-, 0,1)$, where $- : X_n \to X_n$ is restriction of the  opposite map in$\mathbb{Z}$,  $+:X_n\times 
X_n\rightarrow\mathcal{P}(X_n)\setminus\{\emptyset\}$ is given by the rules:
 $$a+b=\begin{cases}
    \{a\},\,\mbox{ if }\,b\ne-a\mbox{ and }|b|\le|a| \\
    \{b\},\,\mbox{ if }\,b\ne-a\mbox{ and }|a|\le|b| \\
    \{-a,...,0,...,a\}\mbox{ if }b=-a
   \end{cases},$$
and $\cdot:X_n\times X_n\rightarrow X_n$ is is given by the rules:
 $$a\cdot b=\begin{cases}
    \mbox{sgn}(ab)\max\{|a|,|b|\}\mbox{ if }a,b\ne0 \\
    0\mbox{ if }a=0\mbox{ or }b=0
   \end{cases}.$$
With the above rules we have that $(X_n,+,\cdot, -, 0,1)$ is a multiring which is not a hyperring for $n\ge2$ because $$n(1-1)=b\cdot\{-1,0,1\}=\{-n,0,n\}$$
and $n-n=X_n$. Note that $X_0=\{0\}$ and $X_1=\{-1,0,1\}\cong Q_2$. 
\end{ex}

\begin{defn}\label{defn:morphism}
 Let $A$ and $B$ multirings. A map $f:A\rightarrow B$ is a morphism if for all $a,b,c\in A$:
 \begin{multicols}{2}
 \begin{enumerate}[i -]
  \item $f(1)=1$ and $f(0)=0$;
  \item $f(-a)=-f(a)$;
  \item $f(ab)=f(a)f(b)$;
  \item $c\in a+b\Rightarrow f(c)\in f(a)+f(b)$.
 \end{enumerate}
 \end{multicols}
  A morphism $f$ is \textbf{a full morphism} if for all $a,b\in A$, $f(a+b)=f(a)+f(b)$.
\end{defn}

\subsection{Special Groups}

Since we will deal with special groups and their relations with hyperfields we provide some basic definitions.

\begin{defn}[Extension of a Relation]\label{extrel}
    Let $A$ be a set and $\equiv$ a binary relation on $A\times A$. We extend $\equiv$ to a binary relation $\equiv_n$ on $A^n$, by induction on $n\ge1$, as follows:
\begin{enumerate}[i -]
\item $\equiv_1$ is the diagonal relation $\Delta_A \subseteq A \times A$
 \item $\equiv_2=\equiv$.
 \item if $n \geq 3$, $\langle a_1,...,a_n\rangle\equiv_n\langle b_1,...,b_n\rangle$ if and only there are $x,y,z_3,...,z_n\in A$ such that 
 \begin{align*}
  &\langle a_1,x\rangle\equiv\langle b_1,y\rangle\\
  &\langle a_2,...,a_n\rangle\equiv_{n-1}\langle x,z_3,...,z_n\rangle\mbox{ and }\\
  &\langle b_2,...,b_n\rangle\equiv_{n-1}\langle y,z_3,...,z_n\rangle
 \end{align*}
\end{enumerate}
\end{defn}

Whenever clear from the context, we frequently abuse notation and indicate the aforedescribed extension $\equiv$ by the same symbol.  

\begin{defn}[Special Group, 1.2 of \cite{dickmann2000special}]\label{defn:sg}
 A \textbf{special group}\index{special group}\index{special group!pre}\index{special group!reduced} is an tuple $(G,-1,\equiv)$, where $G$ is a group of exponent 2, i.e, $g^2=1$ for all $g\in G$; $-1$ is a distinguished element of $G$, and $\equiv\subseteq G\times G\times G\times G$ is a relation (the special relation), satisfying the following axioms for all $a,b,c,d,x\in G$:
\begin{description}
 \item [SG 0] $\equiv$ is an equivalence relation on $G^2$;
 \item [SG 1] $\langle a,b\rangle\equiv \langle b,a\rangle$;
 \item [SG 2] $\langle a,-a\rangle\equiv\langle1,-1\rangle$;
 \item [SG 3] $\langle a,b\rangle\equiv\langle c,d\rangle\Rightarrow ab=cd$;
 \item [SG 4] $\langle a,b\rangle\equiv\langle c,d\rangle\Rightarrow\langle a,-c\rangle\equiv\langle-b,d\rangle$;
 \item [SG 5] $\langle a,b\rangle\equiv\langle c,d\rangle\Rightarrow\langle ga,gb\rangle\equiv\langle gc,gd\rangle,\,\mbox{for all }g\in G$.
 \item [SG 6 (3-transitivity)] the extension of $\equiv$ for a binary relation on $G^3$ (as in \ref{extrel}) is a transitive relation.
\end{description}
\end{defn}

A group of exponent 2, with a distinguished element $-1$, satisfying the axioms SG0-SG3 and SG5 is called a {\bf  proto special group}; a \textbf{pre special group} is a proto special group that also satisfy SG4. Thus a \textbf{special group} is a pre-special group that satisfies SG6 (or, equivalently, for each $n \geq 1$, $\equiv_n$ is an equivalence relation on $G^n$).

A \textbf{$n$-form} (or form of dimension $n\ge1$) is an $n$-tuple of elements of a pre-SG $G$. An element $b\in G$ is \textbf{represented} on $G$ by the form $\varphi=\langle a_1,...,a_n\rangle$, in symbols $b\in D_G(\varphi)$, if there exists $b_2,...,b_n\in G$ such that $\langle b,b_2,...,b_n\rangle\equiv\varphi$. A pre-special group (or special group) 
$(G,-1,\equiv)$ is:\\
$\bullet$ \ \textbf{formally real} if $-1 \notin \bigcup_{n \in \mathbb{N}} D_G( n\langle 1 \rangle)$\footnote{Here the notation $n\langle 1 \rangle$ means the form $\langle a_1,...,a_n\rangle$ where $a_j=1$ for all $j=1,...,n$. In other words, $n\langle 1 \rangle$ is the form $\langle 1 ,...,1\rangle$ with $n$ entries equal to 1.} ;\\
$\bullet$ \ \textbf{reduced} if it is formally real and, for each $a \in G$, $a \in D_G(\langle 1, 1 \rangle)$ iff $a =1$.

\begin{defn}[1.1 of \cite{dickmann2000special}]\label{defnmorph}
 A map $\xymatrix{(G,\equiv_G,-1)\ar[r]^f & (H,\equiv_H,-1)}$ between pre-special groups is a \textbf{morphism 
of pre-special groups or PSG-morphism} if $f:G\rightarrow H$ is a homomorphism of groups, $f(-1)=-1$ and for all 
$a,b,c,d\in G$
$$\langle a,b\rangle\equiv_G\langle c,d\rangle\Rightarrow
\langle f(a),f(b)\rangle\equiv_H\langle f(c),f(d)\rangle$$
A \textbf{morphism of special groups or SG-morphism} is a pSG-morphism between the corresponding pre-special groups. $f$ 
will be an isomorphism if is bijective and $f,f^{-1}$ are PSG-morphisms. 
\end{defn}

It can be verified that a special group  $G$ is formally real iff it admits some SG-morphism $f : G \to 2$. The category of special groups (respectively reduced special groups) and their morphisms will be denoted by $\mbox{SG}$ 
(respectively $\mbox{RSG}$).

\begin{defn}[2.3 of \cite{dickmann2000special}]\label{2.3chico}
 Let $G$ be a special group and let $\Delta\subseteq G$ be a subgroup. We say that $\Delta$ is 
\textit{saturated}\index{saturated subgroup} if for all $a\in G$,
\begin{align*}
 \tag{sat}a\in\Delta\Rightarrow D_G(1,a)\subseteq\Delta.
\end{align*}
Note that if, in addition, $-1\in\Delta$, then $\Delta=G$. Thus we will reserve the noun saturated for those 
subgroups satisfying [sat] such that $-1\notin\Delta$, while $G$ will be called the \textit{improper} 
saturated subgroup of itself.
\end{defn}

\begin{lem}[2.4 of \cite{dickmann2000special}]\label{2.4chico}
 Let $G$ be a special group and $\Delta$ a subgroup of $G$.
 \begin{enumerate}[a -]
  \item The intersection of any family of saturated subgroups is saturated. The union of an upward directed 
family of saturated subgroup is saturated.
  
  \item The following are equivalent:
  \begin{enumerate}[i -]
   \item $\Delta$ is saturated.
   \item For any Pfister forms $\varphi,\psi$ over $\Delta$ and any $b,c\in\Delta$
   $$D_G(\varphi),D_G(\psi)\subseteq\Delta\Rightarrow D_G(b\varphi\oplus c\psi)\subseteq\Delta.$$
   
   \item For any Pfister form $\varphi$ over $\Delta$, $D_G(\varphi)\subseteq\Delta$.
  \end{enumerate}
 \end{enumerate}
\end{lem}


\begin{ex}[The trivial special relation, 1.9 of \cite{dickmann2000special}]\label{ex2.2}
 Let $G$ be a group of exponent 2 and take $-1$ as any element of $G$ different of 1. For $a,b,c,d\in G$, 
define $\langle a,b\rangle\equiv_t\langle c,d\rangle$ if and only if $ab=cd$. Then $G_t=(G,\equiv_t,-1)$ 
is a SG (\cite{dickmann2000special}). In particular $2 = \{-1,1\}$ is a reduced special group.
\end{ex}

\begin{ex}[Special group of a field, Theorem 1.32 of \cite{dickmann2000special}]\label{ex2.3}
 Let $F$ be a field. We denote $\dot F=F\setminus\{0\}$, $\dot F^2=\{x^2:x\in 
\dot F\}$ and $\Sigma \dot F^2=\{\sum_{i\in I}x_i^2:I\mbox{ is finite and }x_i\in 
\dot F^2\}$. Let $G(F)=\dot F/\dot F^2$. In the case of $F$ is be formally real, we have
$\Sigma \dot F^2$ is a subgroup of $\dot F$, then we take $G_{\mbox{red}}(F)=\dot F/\Sigma 
\dot F^2$. Note that $G(F)$ and $G_{\mbox{red}}(F)$ are groups of exponent 2. In 
\cite{dickmann2000special} they prove that $G(F)$ and $G_{\mbox{red}}(F)$ are special groups with the special 
relation given by usual notion of isometry, and $G_{\mbox{red}}(F)$ is always reduced.
\end{ex}

\begin{prop}[3.13 of \cite{ribeiro2016functorial}]\label{sg.to.mf}
 Let $(G,\equiv,-1)$ be a special group and define $M(G)=G\cup\{0\}$ where $0:=\{G\}$\footnote{Here, 
the choice of the zero element was ad hoc. Indeed, we can define $0:=\{x\}$ for any $x\notin G$.}. Then 
$(M(G),+,-,\cdot,0,1)$ is a hyperfield, where 
 \begin{multicols}{2}
 \begin{itemize}
   \item $a\cdot b=\begin{cases}0\,\mbox{if }a=0\mbox{ or }b=0 \\ a\cdot 
b\,\mbox{otherwise}\end{cases}$
   \item $-(a)=(-1)\cdot a$
   \item $a+b=\begin{cases}\{b\}\,\mbox{if }a=0 \\ \{a\}\,\mbox{if }b=0\\ M(G)\,\mbox{if 
}a=-b,\,\mbox{and }a\ne0 
\\ 
D_G(a,b)\,\mbox{otherwise}\end{cases}$
  \end{itemize} 
 \end{multicols}
\end{prop}

\begin{defn}[3.15-3.19 of \cite{ribeiro2016functorial}]\label{smf}
A hyperfield $F$ is a \textbf{special hyperfield} if there exist a special group $G$ such that $F=M(G)$. The category of special hyperfields will be denoted by $\mbox{SMF}$.
\end{defn}

\begin{defn}[Definition 3.2 of \cite{roberto2021quadratic}]
A \textbf{Dickmann-Miraglia multiring (or DM-multiring for short)} \footnote{The name ``Dickmann-Miraglia'' is given in honor to professors Maximo Dickmann and Francisco Miraglia, the creators of the special group theory.} is a pair $(R,T)$ such that $R$ is a multiring, $T\subseteq R$ is a multiplicative subset of $R\setminus\{0\}$, and $(R,T)$ satisfies the following properties:
\begin{description}
\item [DM0] $R/_mT$ is hyperbolic.
 \item [DM1] If $\overline{a}\ne0$ in $R/_mT$, then $\overline a^2=\overline 1$ in $R/_mT$. In other words, for all $a\in R\setminus\{0\}$, 
there are $r,s\in T$ such that $ar=s$. 
 \item [DM2] For all $a\in R$, $(\overline 1-\overline a)(\overline 1-\overline a)\subseteq(\overline 1-\overline a)$ in 
$R/_mT$.
 \item [DM3] For all $a,b,x,y,z\in R\setminus\{0\}$, if 
 $$\begin{cases}\overline a\in \overline x+\overline b \\ \overline b\in \overline y+\overline z\end{cases}\mbox{ in }R/_mT,$$
 then exist $\overline v\in\overline x+\overline z$ such that $\overline a\in\overline y+\overline 
v$ and $\overline{vb}\in\overline{xy}+\overline{az}$ in $R/_mT$.
\end{description}

If $R$ is a ring, we just say that $(R,T)$ is a DM-ring, or $R$ is a DM-ring. A Dickmann-Miraglia hyperfield (or DM-hyperfield) $F$ is a 
hyperfield such that $(F,\{1\})$ is a DM-multiring (satisfies DM0-DM3). In other words, $F$ is a DM-hyperfield if $F$ is hyperbolic and for all 
$a,b,v,x,y,z\in F^*$,
\begin{enumerate}[i -]
 \item $a^2=1$.
 \item $(1-a)(1-a)\subseteq(1-a)$.
 \item $\mbox{If }\begin{cases}a\in x+b \\ b\in y+z\end{cases}\mbox{ then there exists }v\in x+z\mbox{ such that }a\in y+v\mbox{ and }vb\in xy+az$.
\end{enumerate}
\end{defn}

\begin{defn}
A \textbf{hyperbolic multiring} is a multiring $R$ such that $1-1=R$. The category of hyperbolic multirings and hyperbolic hyperfields will be denoted by $\mbox{HMR}$ and $\mbox{HMF}$ respectively. A \textbf{pre-special hyperfield} is a hyperfield satisfying DM0, DM1 and DM2. In other words, a pre-special hyperfield is a hyperbolic hyperfield $F$ such that for all $a\in\dot F$, 
$$a^2=1\mbox{ and }(1-a)(1-a)\subseteq1-a.$$

The category of pre-special hyperfields will be denoted by $PSMF$.
\end{defn}

\begin{teo}[Theorem 3.9 of \cite{roberto2021quadratic}]
 Let $F$ be a hyperfield. Then $F$ is a DM-hyperfield if and only if it is a special hyperfield.
\end{teo}


\begin{prop}\label{psgpsmfhell}
 Let $G$ be a pre-special group and consider $(M(G),+,-,0,1)$, with operations defined by
 \begin{multicols}{2}
 \begin{itemize}
   \item $a\cdot b=\begin{cases}0\,\mbox{if }a=0\mbox{ or }b=0 \\ a\cdot 
b\,\mbox{otherwise}\end{cases}$
   \item $-(a)=(-1)\cdot a$
   \item $a+b=\begin{cases}\{b\}\,\mbox{if }a=0 \\ \{a\}\,\mbox{if }b=0\\ M(G)\,\mbox{if 
}a=-b,\,\mbox{and }a\ne0 
\\ 
D_G(a,b)\,\mbox{otherwise}\end{cases}$
  \end{itemize} 
 \end{multicols}
 Then $M(G)$ is a pre-special hyperfield. Conversely, if $F$ is a pre-special hyperfield then $(\dot F,\equiv_F,-1)$ is a pre-special group, where
 $$\langle a,b\rangle\equiv_F\langle c,d\rangle\mbox{ iff }ab=cd\mbox{ and }a\in c+d.$$
\end{prop}

\begin{prop}\label{pre-krasner}
 Let $F$ be a hyperbolic hyperfield such that $1+1=\{0,1\}$ and $1=-1$. Then $F\cong\{0,1\}$ (the Krasner hyperfield). In particular, $F$ is a DM-hyperfield.
\end{prop}
\begin{proof}
Just observe that
$$F=1-1=1+1=\{0,1\}.$$
\end{proof}

\subsection{Superfields and Polynomials}

We use the concept of superring adopted in (\cite{ameri2019superring}). There are many important advances and results in 
hyperring theory, and 
for instance, we recommend for example, the following papers: \cite{al2019some}, \cite{ameri2017multiplicative}, \cite{ameri2019superring}, \cite{ameri2020advanced}, \cite{massouros1985theory}, \cite{nakassis1988recent}, \cite{massouros1999homomorphic},\cite{massouros2009join}.
 
\begin{defn}[Definition 5 in \textup{\cite{ameri2019superring}}] \label{def:sup}
 A superring is a structure $(S,+,\cdot, -, 0,1)$ such that\footnote{Multigroups and multimonoids are defined in the Appendix.}:
 \begin{enumerate}[i -]
  \item $(S,+, -, 0)$ is a commutative multigroup.
  
  \item $(S,\cdot,1)$ is a commutative multimonoid. 
  
  \item $0$ is an absorbing element: $a\cdot0= \{0\} = 0 \cdot a$, for all $a\in S$.
  
  \item The weak/semi distributive law holds: 
   if $d\in c.(a+b)$ then $d\in (ca+cb)$,
  for all $a,b,c,d\in S$.  
  
  \item  The rule of signals holds:  $-(ab)=(-a)b=a(-b)$, for all $a,b\in S$.
 \end{enumerate}
 A superdomain is a non-trivial superring without zero-divisors in this new context, i.e. whenever
 $$0\in a\cdot b \mbox{ iff }a=0 \mbox{ or } b=0$$
 A quasi-superfield is a non-trivial superring such that every nonzero element is invertible in this new context\footnote{
 For a quasi-superfield $F$, we \textbf{are not imposing} that $(S\setminus\{0\},\cdot,1)$ will be a commutative multigroup, i.e, 
that if $d\in a\cdot b$ then $b^{-1}\in a\cdot d^{-1}$.}, i.e. 
whenever
 $$\mbox{ For all }a \neq 0 \mbox{ exists }b\mbox{ such that }1\in a\cdot b.$$
 A superfield is a quasi-superfield which is also a superdomain. A superring is full if for all $a,b,c,d\in S$, $d\in c\cdot(a+b)$ iff $d\in ca+cb$.
\end{defn}

\begin{ex}
Every multiring can be seen as a superring, in the very same fashion of \ref{ex:1.3}(a). Our main example of superring is the superring of multipolynomials $R[X]$ over a multiring $R$. The construction will be briefly presented in Section \ref{polynomial-section}. For more details, see \cite{roberto2021superrings}, \cite{ameri2019superring} or \cite{davvaz2016codes}.
\end{ex}

For a superfield $F$ and $a\in\dot F$, there exist $b\in\dot F$ with $1\in ab$. Of course, this element $b$ is not necessarily unique. Then, when we write "$b=a^{-1}$" we are choosing some $a^{-1}\in\dot F$ with $1\in a\cdot a^{-1}$.

Now we deal with morphisms.

\begin{defn}
 Let $A$ and $B$ superrings. A map $f:A\rightarrow B$ is a morphism if for all $a,b,c\in A$:
 \begin{multicols}{2}
  \begin{enumerate}[i -]
  \item $f(0)=0$;
  \item $f(1)=1$;
  \item $f(-a)=-f(a)$;
  \item $c\in a+b\Rightarrow f(c)\in f(a)+f(b)$;
  \item $c\in a\cdot b\Rightarrow f(c)\in f(a)\cdot f(b)$.
 \end{enumerate} 
 \end{multicols}
 
 A morphism $f$ is \textbf{a full morphism} if for all $a,b\in A$, $$f(a+b)=f(a)+f(b)\mbox{ and }f(a\cdot b)=f(a)\cdot f(b).$$
\end{defn}

Here, for superrings (and multirings) the terminology "full morphism" has the following inspiration: if $f:R\rightarrow S$ is a full morphism of superrings and $a,b,d\in R$, then $f(d(a+b))=f(d)(f(a)+f(b))$.

\begin{defn}\label{char}
$ $
 \begin{enumerate}[i -] 
 
  \item An \textbf{ideal} of a superring $A$ is a non-empty subset $\mathfrak{a}$ of $A$ such that 
$\mathfrak{a}+\mathfrak{a}\subseteq\mathfrak{a}$ and $A\mathfrak{a}\subseteq\mathfrak{a}$. We denote
  $$\mathfrak I(A)=\{I\subseteq A:I\mbox{ is an ideal}\}.$$
  
  \item  Let $S$ be a subset of a superring $A$. We define the \textbf{ideal generated by} $S$ as 
  $$\langle S\rangle:=\bigcap\{\mathfrak{a}\subseteq A\mbox{ ideal}:S\subseteq\mathfrak{a}\}.$$
  If $S=\{a_1,...,a_n\}$, we easily check that
  $$\langle a_1,...,a_n\rangle=\sum Aa_1+...+\sum Aa_n,\,\mbox{where }\sum 
Aa=\bigcup\limits_{n\ge1}\{\underbrace{Aa+...+Aa}_{n\mbox{ times}}\}.$$ 
  Note that if $A$ is a full superring, then $\sum Aa=Aa$.
  
  \item An ideal $\mathfrak{p}$ of $A$ is said to be \textbf{prime} if $1\notin\mathfrak{p}$ and 
$ab\subseteq\mathfrak{p}\Rightarrow a\in\mathfrak{p}$ or $b\in\mathfrak{p}$. We denote 
  $$\mbox{Spec}(A)=\{\mathfrak{p}\subseteq A:\mathfrak{p}\mbox{ is a prime ideal}\}.$$
  
  \item An ideal $\mathfrak{p}$ of $A$ is said to be \textbf{strongly prime} if $1\notin\mathfrak{p}$ and 
$ab\cap\mathfrak{p}\ne\emptyset\Rightarrow a\in\mathfrak{p}$ or $b\in\mathfrak{p}$. We denote 
  $$\mbox{Spec}_s(A)=\{\mathfrak{p}\subseteq A:\mathfrak{p}\mbox{ is a strongly prime ideal}\}.$$
  Note that every strongly prime ideal is prime. Moreover, if $A$ is a multiring, the notions os prime and strong prime ideal coincide.
  
  \item An ideal $\mathfrak{m}$ is maximal if it is proper and for all ideals $\mathfrak{a}$ with 
  $\mathfrak{m}\subseteq\mathfrak{a}\subseteq 
  A$ then $\mathfrak{a}=\mathfrak{m}$ or $\mathfrak{a}=A$.
  
  \item For an ideal $I\subseteq A$, we define operations in the quotient $A/I=\{x+I:x\in A\}=\{\overline x:x\in A\}$, by the 
rules
  \begin{align*}
      \overline x+\overline y&=\{\overline z:z\in x+y\}\\
      \overline x\cdot\overline y&=\{\overline z:z\in xy\}
  \end{align*}
    for all $\overline x,\overline y\in A/I$.
 \end{enumerate}
\end{defn}

With all conventions and notations, we obtain the following Lemma, which recover for superrings some properties holding for rings (and multirings).

\begin{lem}\label{lem1}
 Let $A$ be a superring and $I$ an ideal.
 \begin{enumerate}[i -]
  \item $I=A$ if and only if $1\in I$.
  \item $A/I$ is a superring. Moreover, if $A$ is full then $A/I$ is also full.
  \item $I$ is strongly prime if and only if $A/I$ is a superdomain.
 \end{enumerate}
 If $A$ is full, then
 \begin{enumerate}
     \item [iv -] $I=A$ if and only if $1\in I$, which occurs if and only if $A^*\cap I\ne\emptyset$ (in other words, if and only if 
$I$ contains an invertible element).
     \item [v -] $A$ is a superfield if and only if $\mathfrak I(A)=\{0,A\}$.
     \item [vi -] $I$ is maximal if and only if $A/I$ is a superfield.
 \end{enumerate}
\end{lem}
\begin{proof}
    Items (i), (ii) and (iii) are immediate. For (iv), just note that $\overline 0\in\overline{ab}$ if and only if $ab\cap 
I\ne\emptyset$. Item (v) is an immediate consequence of items (i)-(iv).
\end{proof}

\begin{prop}
 Let $A$ be a superring and $I$ an ideal.
 \begin{enumerate}[i -]
     \item If $I$ is a maximal ideal, then it is prime.
     
     \item The ideal $I$ is prime if, and only if, $A/I$ is  quasi-superdomain\footnote{A superring $B$ is called quasi-superdomain if given $a,b \in B$ with $ab = \{0\}$, then $a = 0$ or $b = 0$}.
     
     \item (Prime Ideal Theorem) Let $S \subseteq A$ be a multiplicative set ($1 \in S$ and $S \cdot S \subseteq S$). Suppose that $S \cap I = \emptyset$. Then there is a prime ideal $p$ such that $I \subseteq p$ and $S \cap p = \emptyset$.
 \end{enumerate}
\end{prop}

\begin{proof}
$ $
 \begin{enumerate}[i -]
     \item Let $a,b \in A$ with $ab \subseteq I$. Assume that $a \notin I$ and consider the ideal $J = I + (a)$. Then there are $x \in I, t_1, \ldots, t_n \in A$ such that $1 \in x + at_1 + \cdots + at_n$. Thus $b \in bx + bat_1 + \cdots + bat_n \subseteq I$.
     
     \item If $I$ is prime and $\overline{a} \cdot \overline{b} = \{\overline{0}\}$ in $A/I$, then $a\cdot b \subseteq I$. Therefore, by primality, $a \in I$ or $b \in I$. Thus $\overline{a} = 0$ or $\overline{b} = 0$ in $A/I$.
     Conversely, assume $A/I$ a quasi-superdomain and let $a,b \in A$ with $ab \subseteq I$. Then $\overline{a} \cdot \overline{b} = \{\overline{0}\}$ and by hypothesis follows $a \in I $ or $b \in I$, as desired.
     
     \item Consider the partial order $\mathcal{X} = \{J \subseteq A \colon J\mbox{ is an ideal and } S \cap J = \emptyset\}$ ordered by inclusion. Since the directed union of ideals is again an ideal, we have by Zorn's Lemma that $\mathcal{X}$ has a maximal element $p$. Suppose that $p$ is not prime, that it, there is $a,b \in $ with $ab \subseteq p$ and $a,b \notin p$. Now notice that $J_a = p + (a)$ and $J_b = p + (b)$ are ideals that proper extend $p$. Hence, by maximality, there are $s,v \in S, x,y \in p, t_1, \ldots, t_n, w_1, \ldots, w_k \in A$ with 
     \begin{align*}
         s &\in x + at_1 + \cdots + at_n \\
         v &\in y + bw_1 + \cdots + bw_k.
     \end{align*}
     
     These equations implies that 
     $$sv \subseteq xy + xbw_1 + \cdots + xbw_k + yat_1 + \cdots + yat_n + \sum_{i,j}abt_iw_j \subseteq S \cap p,$$
     a contradiction. Then $p$ is prime.
 \end{enumerate}
\end{proof}

Even if the multivalued structures similar to rings have been studied for more than 70 years, the developments of notions  of polynomials in the ring-like multivalued structure seems to have a more significant development only from the last decade: for instance in \cite{jun2015algebraic} some notion of multi polynomials is introduced to obtain some applications to algebraic and tropical geometry, in \cite{ameri2019superring} a more detailed account of variants of concept of multipolynomials over hyperrings is applied to get a form of Hilbert's Basissatz.


Here we will stay close to \cite{ameri2019superring} perspective: let $(R,+,-,\cdot,0,1)$ be a \emph{full superring} and set
$$R[X]:=\{(a_n)_{n\in\omega}\in R^\omega:\exists\,t\,\forall n(n\ge t\rightarrow a_n=0)\}.$$
Of course, we define the \textbf{degree} of $(a_n)_{n\in\omega}$ to be the smallest $t$ such that $a_n=0$ for all $n>t$. 

Now define the binary multioperations $+,\cdot :  R[X]\times R[X] \to  \mathcal P^*(R[X])$, a unary operation $-:R[X]\rightarrow R[X]$ and elements $0,1\in R[X]$ by
\begin{align*}
 (c_n)_{n\in\omega}\in (a_n)_{n\in\omega}+(b_n)_{n\in\omega}&\mbox{ iff }\forall\,n(c_n\in a_n+b_n) \\
  (c_n)_{n\in\omega}\in((a_n)_{n\in\omega}\cdot (b_n)_{n\in\omega}&\mbox{ iff }\forall\,n
  (c_n\in a_0\cdot b_n+a_1\cdot b_{n-1}+...+a_n\cdot b_0) \\
  -(a_n)_{n\in\omega}&=(-a_n)_{n\in\omega} \\
  0&:=(0)_{n\in\omega} \\
  1&:=(1,0,...,0,...)
\end{align*}
For convenience, we write elements of $R[X]$ by $\alpha=(a_n)_{n\in\omega}$. Beside this, we denote
\begin{align*}
 1&:=(1,0,0,...), \\
 X&:=(0,1,0,...), \\
 X^2&:=(0,0,1,0,...)
\end{align*}
etc. In this sense, our ``monomial'' $a_iX^i$ is denoted by $(0,...0,a_i,0,...)$, where $a_i$ is in the $i$-th position; in 
particular, we will write ${\underline{b}} = (b,0,0,...)$ and we frequently identify $b \in R \leftrightsquigarrow 
{\underline{b}} \in R[X]$.

 The fact that $R[X]$ is a superring is already stated in Theorem 2 in \cite{ameri2019superring}.

\begin{teo}\label{teo-rxfull}
    Let $R$ be a full superring. Then $R[X]$ is a superring.
\end{teo}
\begin{proof}
    In fact, the only non trivial property we need to prove for $R[X]$ is the associativity of product. To prove it, we deal with elements in $R[X]$ as sequences: we denote $a=(a_0,a_1,....,a_n,...)\in R[X]$, and for $n\ge0$, $[a]_n:=a_n$. We extend this notation for the operations $+$ and $\cdot$ over $R[X]$:
    \begin{align*}
        [a+b]_n&:=a_n+b_n\\
        [ab]_n&:=\sum^n_{i=0}a_ib_{n-i}
    \end{align*}
    With these notations, for all $a,b,c\in R[X]$ and all $n\ge0$ we get
    \begin{align*}
        [(ab)c]_n&= \sum^n_{i=0}\left[[ab]_i\right]c_n
        =\sum^n_{i=0} \left[\sum^i_{p=0}a_pb_{i-p}\right]c_n\\
        [a(bc)]_n&= \sum^n_{i=0}a_i\left[[bc]_{n-i}\right]=\sum^n_{i=0}a_i \left[\sum^{n-i}_{p=0}b_pc_{n-i-p}\right]
    \end{align*}
    If $R$ is full, then (after some re-indexation) we get $[(ab)c]_n=[a(bc)]_n$ for all $n\ge0$, and then, $(ab)c=a(bc)$.
\end{proof}

The properties stated in the Lemma below immediately follows from the definitions involving $R[X]$:

\begin{lem}[Basic Facts about Polynomials]\label{lemperm}
 Let $R$ be a full superring and $R[X]$ as above and $n,m\in \mathbb N$.
 \begin{enumerate}[a -]
 \item $R[X]$ is a superdomain iff $R$ is a superdomain.
 \item The map $a \in R \mapsto {\underline{a}} = (a,0, \cdots,0, \cdots)$ defines a full embedding $R\rightarrowtail R[X]$.
  \item For an ordinary ring $R$ (identified with a strict superring), the superring $R[X]$ is naturally isomorphic to (the superring associated to) the ordinary ring of polynomials in one variable over $R$.
  \item $X^n\cdot X^m=\{X^{n+m}\}$.
  \item For all $a\in R$, ${\underline{a}}\cdot X^n=\{aX^n\}$.
  \item Given $\alpha=(a_0,a_1,...,a_n,0,0,...)\in R[X]$, with with $\deg\alpha \leq n$ and $m\ge1$, we have
  $$\alpha X^m=(0,0,...,0,a_0,a_1,...,a_n,0,0,...)=a_0X^m+a_1X^{m+1}+...+a_nX^{m+n}.$$
  \item For $\alpha=(a_n)_{n\in\omega}\in R[X]$, with $\deg\alpha=t$, 
  $$\{\alpha\}=a_0\cdot1+a_1\cdot X+...+a_t\cdot X^t=a_0+X(a_1+a_2X+...+a_nX^{t-1}).$$
  \item For all $a,b\in R$, $X^k(a + b) = aX^k+ bX^k$.
   \end{enumerate}
\end{lem}

Lemma \ref{lemperm} allow us to deal with the superring $R[X]$ as usual. In other words, we can assume that for $\alpha\in R[x]$, there exists 
$a_0,a_1,...,a_n\in R$ such that $\alpha=a_0+a_1X+...+a_nX^n$, and then, we can work simply denoting $\alpha=f(X)$, as usual. For example, combining the definitions and all facts above we get
$$(x-a)(x-b)=x^2+(-a-b)x+ab=\{x^2+dx+e:d\in -a-b\mbox{ and }e\in ab\}.$$

\begin{rem} 
If $R$ is a full superdomain, does not hold in general that $R[X]$ is also a full superdomain. In fact, even if $R$ is a 
hyperfield, there are examples, e.g. $R = K, Q_2$, such that $R[X]$ is not a full superdomain (see \cite{ameri2019superring}).
\end{rem}

\begin{defn}
 The superring $R[X]$ will be called the \textbf{superring of polynomials} with one variable over $R$. The elements of $R[X]$ will be called \textbf{polynomials}.
\end{defn}

\begin{lem}[Adapted from Theorem 5 of \cite{ameri2019superring}]\label{degreelemma}
Let $R$ be a full superring and $f,g\in R[X]\setminus\{0\}$.
\begin{enumerate}[i -]
    \item If $t(X)\in f(X)+g(X)$ and $f\ne-g$ then $$\min\{\deg(f),\deg(g)\}\le\deg(t)\le\max\{\deg(f),\deg(g)\}.$$
    \item If $R$ is a superdomain and $t(X)\in f(X)g(X)$, then $\deg(t)=\deg(f)+\deg(g)$. In particular, if $f_1(X),f_2(X),...,f_n(X)\ne0$ and $t(X)\in f_1(X)f_2(X)...f_n(X)$, then
    $$\deg(t)=\deg(f_1)+\deg(f_2)+...+\deg(f_n).$$
    \item (Partial Factorization) Let $R$ be a superdomain with $f\in (X-a_1)(X-a_2)...(X-a_p)$. Then $\deg(f) = p$.
\end{enumerate}
\end{lem}

\begin{teo}[Euclid's Division  Algorithm (3.4 in \cite{davvaz2016codes})]\label{euclid}
 Let $K$ be a full superfield. Given polynomials $f(X),g(X)\in K[X]$ with $g(X)\ne0$, there exists $q(X),r(X)\in 
K[X]$ such that $f(X)\in q(X)g(X)+r(X)$, with $\deg r(X)<\deg g(X)$ or $r(X)=0$.
\end{teo}

\begin{teo}[Adapted from Theorem 6 of \cite{ameri2019superring}]\label{teoPID}
 Let $F$ be a full superfield. Then $F[X]$ is a principal ideal superdomain.
\end{teo}
\begin{proof}
Let $I$ be a ideal of $F[X]$. If $I=0$ then $I=\langle0\rangle$ and if there is some $a\in F\setminus\{0\}$ with $a\in I$, then $I=F[X]=\langle 1\rangle$ (because $F$ is full).

Now let $p(X)\in I$ be a polynomial with minimal degree $m\ge1$. Let $f(X)\in I$ be another polynomial. By Euclid's Algorithm, there exists $q(X),r(X)\in F[X]$ with $f(X)\in p(X)q(X)+r(X)$ and $r(X)=0$ or $\deg(r)<\deg(p)=m$. Since $f,p\in I$ and $r(X)\in f(X)-p(X)q(X)$, we have $r\in I$. Note that by the minimality of $m$, all nonzero polynomial in $f(X)-p(X)q(X)$ has degree at least $m$. If $r\ne0$ then 
$$\min\{\deg f,\deg(p)+\deg(q)\}\le\deg r\le\max\{\deg f,\deg(p)+\deg(q)\}.$$
In particular $\deg(r)\ge m$ (because $\deg(f)\le m$), contradicting $deg(r)<m$. Hence $r=0$ and $I=\langle p\rangle$. In particular, $I=F[X]\cdot p(X)$.
\end{proof}

Let $R, S$ be full superrings and $h : R \to S$ be a morphism. Then $h$ extends naturally to a morphism in the superrings multipolynomials $h^X : 
R[X] \to S[X]$:
$$(a_n)_{n \in \mathbb{N}} \in R[X] \ \mapsto \ (h(a_n))_{n \in \mathbb{N}} \in S[X]$$

Now let $s \in S$. We define the $h$-\textbf{evaluation} of $s$ at $f(X)\in R[X]$ with $f(X)=a_0+a_1X+...+a_nX^n$ by
$$f^h(s)=ev^h(s,f):=\{s'\in S : s'\in h(a_0)+ h(a_1).s+h(a_2).s^2+...+h(a_n).s^n\}.$$
We define the $h$-\textbf{evaluation} for a subset $I\subseteq S$ by
$$f^h(I)=\bigcup_{s\in I}f^h(s).$$

In particular if $S\supseteq R$ are full superrings and $\alpha\in S$, we have the \textbf{evaluation} of $\alpha$ at  $f(X)\in R[X]$ by
$$f(\alpha,S)=ev(\alpha,f,S)=\{b\in S: b \in a_0+a_1\alpha+a_2\alpha^2+...+a_n\alpha^n\}\subseteq S.$$
When $S=R$ we just write $f(\alpha,R)$ by $f(\alpha)$. 

A \textbf{root} of $f$ in $S$ is an element $\alpha\in S$ such that $0\in ev(\alpha,f,S)$. In this case we say that $\alpha$ is \textbf{$S$-algebraic} over $R$. An \textbf{effective root} of $f$ in $S$ is an element $\alpha\in S$ such that $f\in(X-\alpha)\cdot g(X)$ for some $g(X)\in R[X]$.

Observe that, if $F$ is a field, the evaluation of $F[X]$ as a ring coincide with the usual evaluation, and, of course, root and effective roots are the same thing. Therefore, if $F$ is algebraically closed as hyperfield and superfield, then will be algebraically closed in the usual sense. 

\begin{ex}[Polynomials can have infinite roots]\label{inftyroots}
Let $F$ be a infinite pre-special hyperfield. Then $F$ has characteristic $0$, $a^2=1$ for all $a\ne0$ so the polynomial $f(X)=X^2-1$ has infinite roots (i.e, $0\in ev(f,\alpha)$ for all $\alpha\in\dot F$).
\end{ex}

The next result say that for hyperfields, every root is in fact an effective root:

\begin{teo}[Adapted from Theorem 7 in \cite{ameri2019superring}]\label{effect}
 Let $F$ be a hyperfield. If $\alpha$ is a root of $f(X)\in F[X]$ then there exist $g(X)\in F[X]$ such that $f(x)\in 
(X-\alpha)\cdot g(X)$.
\end{teo}



\begin{defn}[Irreducibility]
Let $R$ be a full superfield and $f,d\in R[X]$. We say that $d$ divides $f$ if and only if $f\in\langle d\rangle$, and denote $d|f$. 
We say that $f$ is \textbf{irreducible} if $\deg f\ge1$ and $u|f$ for some $u\in R[X]$ (i.e, $f\in\langle u\rangle$), then 
$\langle f\rangle=\langle u\rangle$.\\
\end{defn}

\begin{teo}\label{lemquadext2}
Let $F$ be a full superfield and $p(X)\in F[X]$ be an irreducible polynomial. Then $\langle p(X)\rangle$ is a maximal ideal. In particular, $F[X]/\langle p(X)\rangle$ is a (full) superfield. We write $$F(p):=F(p(X))=F[X]/\langle p(X)\rangle.$$
\end{teo}
\begin{proof}
Let $p(X)$ be irreducible and $I\subseteq F[X]$ an ideal with $\langle p(X)\rangle\subseteq I$. By Theorem \ref{teoPID}, 
   $$I=\langle f(X)\rangle=F[X]\cdot f(X)$$
   for some $f(X)\in F[X]$. Since $p(X)\in I=\langle f(X)\rangle$, then $p(X)\in f(X)g(X)$ for some $g(X)\in F[X]$. Since $p(X)$ is irreducible, either $f(X)$ or $g(X)$ is a constant polynomial. If $f(X)$ is constant, then $I=F[X]$, and if $g(X)$ is constant, $I=\langle p(X)\rangle$, which proves that $\langle p(X)\rangle$ is maximal.

   Note that item (vi) in Lemma \ref{lem1} provide that $F[X]/\langle p(X)\rangle$ is a full superfield.
\end{proof}

\section{A New proof of Hauptsatz for Reduced Special Groups}\label{newresg}

In this Section, we present an alternative proof of APH for reduced special groups, using multivalued structures and inspired in the machinery developed for M. Marshall in \cite{marshall2006real}. This alternative proof avoid the use of the invariants developed in \cite{dickmann2000special}. Moreover, we obtain a glance of the necessary new methods to attack the APH for special groups in general, which will be developed in the next Sections.

$ $

In the realm of multirings, the notion of the so called "Marshall's  quotient", introduced in \cite{marshall2006real} and further developed in \cite{ribeiro2016functorial}, is a quotient multiring  defined for pair $(A,S)$ where $A$ is a multiring and $S \subseteq A$ is a multiplicative subset\footnote{For a brief exposition of the theories of special groups, multirings, hyperfields and its relations the reader can look at the Appendix \ref{preliminaries-section}.}: given $a, b \in A $, 
$$a \approx_S b \ \mbox{ iff there are}\ x,y \in S \ \mbox{such that}\  ax =by.$$

Now we introduce the following: let $G$ be a pre-special group and $T\subseteq G$ be a \textbf{preordering}, i.e, a multiplicative subset ($T\cdot T\subseteq T$) with $1\in T$ and for all $a,b\in T$, $D_G(a,b)\subseteq T$. Define an equivalence relation on $G$ by
$$g\sim h\mbox{ iff }gr=hs\mbox{ for some }r,s\in T.$$
Denote $G/_mT:=G/\sim$ and denote by $[g]$ the equivalence class of $g\in G$. We define
$$\langle[a],[b]\rangle\equiv_{G/_mT}\langle[c],[d]\rangle\mbox{ iff }[ab]=[cd]\mbox{ and }acr\in D_G(1,cds)\mbox{ for some }r,s\in T.$$

\begin{teo}\label{teocdpe}
    With the above prescriptions, $G/_mT$ is a reduced special group, called the \textbf{Marshall quotient of $G$ by $T$}.
\end{teo}
\begin{proof}    
    The proof is easier in the language of hyperfields: using the machinery of 3.13-3.18 of \cite{ribeiro2016functorial} (or 3.2-3.10 of \cite{roberto2021quadratic}) we interpret $F:=G\cup\{0\}$ as a hyperfield and $D_G(1,a)$ as $1+a$. In Proposition 3.17 \cite{ribeiro2016functorial} in conjunction with Theorem 3.3 of \cite{ribeiro2016functorial} and the machinery developed in Section 4 of \cite{marshall2006real} we conclude that: i) every reduced special group (interpreted as a hyperfield) is a real reduced hyperfield ($a^2=1$, $1\ne-1$ and $1+1=\{1\}$); ii) if a pre-special group $G$ satisfies [SG0-SG5] and [red], then $G$ satisfy [SG6], and hence, $G$ is a reduced special group.
    
    Then, in our case we have that $F/_mT\cong\{0,1\}$ or $F/_mT$ is a real reduced hyperfield, which means that $G/_mT$ is a reduced special group.
\end{proof}

The Marshall quotient of reduced special groups has a universal property, which is similar to the Marshall quotient of hyperfields:

\begin{teo}[Proposition 2.19 of \cite{ribeiro2016functorial}]\label{mquotup}
    Let $G$ be a special group and $T\subseteq G$ be a multiplicative subset with $1\in T$. The projection map $\pi:G\rightarrow G/_mT$ is a universal morphism satisfying $\pi(T) = \{1\}$, that is, given a morphism $f:G\rightarrow H$ with $f(T)= \{1\}$, there is an unique morphism $[f]:G/_mT\rightarrow H$ such that $f=[f]\circ\pi$. In other words, for every morphism $f:G\rightarrow H$ such that $f[S]=\{1\}$, there exist a unique morphism $[f]:G/_mT\rightarrow H$ such that the following diagram commute:
    $$\xymatrix@!=4.5pc{G\ar[r]^{\pi}\ar[dr]_{f} & G/_mT\ar[d]^{![f]} \\ & H}$$
\end{teo}

In fact, if $G$ is a reduced special group and $H\subseteq G$ is a Pfister subgroup (see Definition 2.15 of \cite{dickmann2000special}), then the Marshall quotient $G/_mH$ is isomorphic to the Pfister quotient of $G$ by $H$ (just compare our structure to the one defined in Sections 2 and 3 of Chapter 2 in \cite{dickmann2000special}).

\begin{defn}
    Let $G$ be a reduced special group and $a\in G$. We simply denote 
    $$G(a):=G/_m D_G(1,a).$$
\end{defn}

For the benefit of the reader we prove the following Theorem.

\begin{teo}
    Let $G$ be a reduced special group and $a\in G$. Then $G(a)$ is a reduced special group.
\end{teo}
\begin{proof}
    By Theorem \ref{teocdpe} we only need to prove that $D_G(1,a)$ is a preordering. The proof is easier in the language of hyperfields: using the machinery of Proposition 3.15 of \cite{ribeiro2016functorial} (or Theorem 3.9 of \cite{roberto2021quadratic}) we interpret $D_G(1,a)$ as $1+a$. In Proposition 3.15 of \cite{ribeiro2016functorial} is stated that $D_G(1,a)$ is multiplicative, then $(1+a)(1+a)\subseteq(1+a)$. Now, observe that
    $$(1+a)+(1+a)=1\cdot(1+a)+a\cdot(1+a)=(1+a)(1+a)\subseteq(1+a),$$
    proving that $1+a=D_G(1,a)$ is a preordering of $G$.
\end{proof}

\begin{teo}\label{sext1red}
    Let $G$ be a reduced special group and $a,b\in G$. Then
    $$G(a)(b)\cong G(b)(a).$$
\end{teo}
\begin{proof}
    For instance, denote the elements of $G(a),G(b),G(a)(b)$ and $G(b)(a)$ respectively by $[g]_{a},[g]_{b},[g]_{(a,b)}$ and $[g]_{(b,a)}$, $g\in G$.

    Let $\varphi:G\rightarrow G(a)(b)$ given by $\varphi(g)=[g]_{(a,b)}$. Then $\varphi$ is a morphism such that $\varphi[D_G(1,b)]=\{[1]_{(a,b)}\}$. By the universal property \ref{mquotup} there is an unique morphism $$\varphi_{b}:G/_mD_G(1,b)=G(b)\rightarrow G(a)(b)$$
    such that $\varphi_b\circ\pi=\varphi$, which is given by $\varphi_{b}([g]_b)=[g]_{(a,b)}$. Moreover, $\varphi_{b}[D_{G(b)}(1,a)]=\{[1]_{(b,a)}\}$, which imply (by the universal property \ref{mquotup}) that there exists an unique morphism    $$\varphi_{(b,a)}:G(b)/_mD_{G(b)}(1,a)=G(b)(a)\rightarrow G(a)(b)$$
    such that $\varphi_{(b,a)}\circ\pi=\varphi_b$, which is given by $\varphi_{(b,a)}([g]_{(b,a)})=[g]_{(a,b)}$.

    Similarly, we get an unique morphism 
    $$\psi_{(a,b)}:G(a)(b)\rightarrow G(b)(a).$$
    By the universal property we get that $\varphi_{(b,a)}$ and $\psi_{(a,b)}$ are isomorphisms, inverse of each other.
\end{proof}

With Theorem \ref{sext1red} we are able to properly iterate the construction $G(a)$. For $a_1,...,a_{n+1}\in G$ with $n\ge1$, we define recursively:
\begin{align*}
    G(a_1,...,a_{n+1}):=G(a_1,...,a_n)(a_{n+1}).
\end{align*}

\begin{lem}
    Let $G$ be a reduced special group, $a_1,...,a_n\in G$, and $\sigma\in S_n$. Then
    $$G(a_{\sigma(1)},...,a_{\sigma(n)})\cong G(a_1,...,a_n).$$
\end{lem}
\begin{proof}
    Just use \ref{sext1}, the fact that $S_n$ is generated by transpositions and induction.
\end{proof}

In the sequel, our objective is characterise for a reduced special group $G$, whether $[a]=[b]$ in $G(a_1,...,a_n)$ in terms of the equations in $G$. 

\begin{teo}\label{iso0red}
 Let $G$ be a reduced special group and $a,b,c\in G$. Then $[a]=[b]$ in $G(a)$ iff there is $s,t\in D_G(1,a)$ with $as=bt$ (or $a=bst$ or even $b=ats$).
\end{teo}
\begin{proof}
    Just use the Definition of $G(a)$ and the Marshall quotient.
\end{proof}

\begin{teo}\label{teo32red}
    Let $G$ be a reduced special group and $a_1,...,a_n,b,c\in G$. Then $[a]=[b]$ in $G(a_1,...,a_n)$ iff there is $s,t\in D_G(\langle\langle a_1,a_2,...,a_n\rangle\rangle)$ such that $as=bt$ (or $a=bst$ or even $ab\in D_G(\langle\langle a_1,a_2,...,a_n\rangle\rangle)$).
\end{teo}
\begin{proof}
    We proceed by induction. The case $n=1$ is just Theorem \ref{iso0red} (since $\langle\langle a\rangle\rangle=\langle1,a\rangle$). Now suppose the desired valid for $n$, and let $[a]=[b]$ in $G(a_1,...,a_{n+1})$. Since $$G(a_1,...,a_{n+1})\cong G(a_1,...,a_n)(a_{n+1}),$$
    by Theorem \ref{iso0red} we have $[ar]=[bs]$ for some $[r],[s]\in D_{G(a_1,...,a_n)}(1,[a_{n+1}])$. By induction hypothesis, we get $art=bsw$ with $[r],[s]\in D_{G(a_1,...,a_n)}(1,[a_{n+1}])$ and $t,w\in D_G(\langle\langle a_1,...,a_n\rangle\rangle)$.

    Now, observe that $[r]\in D_{G(a_1,...,a_n)}(1,[a_{n+1}])$ means $[r]=[r'a_{n+1}]$ for some $[r']\in G(a_1,...,a_n)$. By induction hypotesis, $ru=r'a_{n+1}v$ for some $u,v\in D_G(\langle\langle a_1,...,a_n\rangle\rangle)$. Then
    \begin{align*}
        atu&=a(rt)(ru)=b(sw)(r'a_{n+1}v)=b(swr'a_{n+1}v),
    \end{align*}
    with $tu,swr'a_{n+1}v\in D_G\langle\langle a_1,a_2,...,a_{n+1}\rangle\rangle$ (here we use that Pfister forms are multiplicative).
\end{proof}

\begin{cor}\label{cor33red}
  Let $G$ be a reduced special group and $a_1,a_2,...,a_n\in G$. Then $G(a_1,a_2,...,a_n)$ is formally real iff $-1\notin D_G(\langle\langle a_1,a_2,...,a_n\rangle\rangle)$.
\end{cor}
\begin{proof}
Since $G(a_1,a_2,...,a_n)$ is a reduced special group, we have $G(a_1,a_2,...,a_n)$ formally real iff $[1]\ne[-1]$, which by Theorem \ref{teo32red} occurs iff $-1\notin D_G(\langle\langle\alpha_1,\alpha_2,...,\alpha_n\rangle\rangle)$.
\end{proof}

 We start establishing the following:

 \begin{rem}[Notations and Facts]\label{pfisternotation}
 Here we summarize some notations about Witt equivalence, Pfister forms and fundamental ideal (for more details, see for instance \cite{dickmann2000special} or \cite{lam2005introduction}):
     \begin{itemize}
    \item Let $\varphi,\psi$ be forms on a special group $G$. We say that $\varphi$ and $\psi$ are \textup{Witt equivalent}, notation $\varphi\approx_{W,G}\psi$ iff there exist non negative integers $k,l$ such that $k\langle1,-1\rangle\oplus\varphi\equiv_Gl\langle1,-1\rangle\oplus\psi$. By Witt's Decomposition, if $\varphi$ is a form on $G$, there are unique forms $\varphi_{an},\varphi_{hip},\varphi_{0}$ (up to isometry) with $\varphi\equiv\varphi_{an}\oplus\varphi_{hip}\oplus\varphi_{0}$, $\varphi_{an}$ anisotropic, $\varphi_{hip}$ hyperbolic and $\varphi_{0}$ totally isotropic. We define $\dim_{W,G}(\varphi):=\dim(\varphi_{an})$.
    
    \item  Let $G$ be a special group. For each $n \in \mathbb{N}$ consider the statement:\\
    $AP_G(n)$ For each $\varphi = \langle a_1,\cdots, a_k \rangle$, a  non-empty ($k\geq 1$), regular ($a_i \in G$) and anisotropic form, if  $\varphi\in I^n(G)$, then $\dim(\varphi)\ge2^n$.

    \item A Pfister form of degree $n \geq 1$, with coefficients $a_1, \cdots, a_n \in G$ is $\langle\langle a_1, \cdots, a_n \rangle \rangle = \otimes_{i=1}^{n} \langle 1, a_i\rangle$.
    
    \item Let  $\psi$ be a Pfister form. Then $\psi$ is hyperbolic iff it is isotropic. Moreover, if $G$ is reduced and $-1 \in D_G(\psi)$, then $\psi$ is hyperbolic. 

    \item  $I^n(G) \subseteq W(G)$ is additively generated by the Pfister forms of degree $n$. 
    
    \item  If $\varphi \in I^n(G) \setminus\{\emptyset\}$, then $\varphi = \varepsilon_1\varphi_1+...+\varepsilon_r\varphi_r$, where  $r\ge1$ and $\varepsilon_j=\pm1$ for all $j=1,...,r$. Moreover, if $\varphi$ is anisotropic, we suppose without loss of generality that $\varepsilon_j=1$ for all $j=1,...,r$.
\end{itemize}
 \end{rem}
 
\begin{teo}[Arason-Pfister Hauptsatz]\label{haupred}
 Let $G$ be a reduced special group, then  it holds $AP_G(n)$, for all $n \geq 0$. In more details: for each  $n \geq 0$ and For each $\varphi = \langle a_1,\cdots, a_k \rangle$, a  non-empty ($k\geq 1$), regular and anisotropic form, if $\varphi\in I^n(G)$, then $\dim(\varphi)\ge2^n$  $\varphi\in I^n(G)$, if $\varphi \neq \emptyset$ is anisotropic, then $\dim_{W,G}(\varphi)\ge2^n$.
\end{teo}

\begin{proof}

An equivalent way to state this result is the following: if a form $q$ belongs to $I^nG$ and $\dim q<2^n$ then $q$ must be a hyperbolic form.

 Since $\varphi$ is an anisotropic form such that $\varphi \in I^nG \setminus \{\emptyset\}$ and  $I^nG$ is additively generated by the Pfister forms, then there exists $r \geq 1$ and Pfister forms of degree $n$, $\varphi_1, \cdots \varphi_r$ such that $\varphi= \pm(\varphi_1+...+\varphi_r)$.  

Since $\varphi$ is anisotropic, we can suppose without loss of generality that $\varphi= \varphi_1+...+\varphi_r$ and proceed by induction on $r$. 

If $r=1$, then $\varphi=\varphi_1$, with $\dim(\varphi)=\dim(\varphi_1)=2^n$.

Let $r\ge2$. If $\varphi_j$ is isotropic for all $j=1,...,r$ then $\varphi$ is isotropic (hyperbolic, in fact): this fallows from Witt's cancellation law since $\varphi \oplus k \langle 1, -1 \rangle \equiv (r2^{n-1}+m) \langle 1, -1 \rangle $. So we can suppose without loss of generality that $\varphi_1=\langle\langle a_1,...,a_n\rangle\rangle$ is anisotropic.

Suppose $-1\notin D_G(\varphi_1)$. Let $G(\varphi_1):=G(a_1,...,a_n)$. Then equivalence class of $\varphi$ on $G(\varphi_1)$ is
$$[\varphi]=[\varphi_1+...+\varphi_r]=[\varphi_1]+...+[\varphi_r]=2^n\langle1\rangle+[\varphi_2]+...+[\varphi_r].$$
We already know that $\dim_{W,G}(\varphi)\ge\dim_{W,G(\varphi_1)}[\varphi]$. Then we have three cases:
\begin{description}
\item [I - ] $[\varphi]$ is hyperbolic. Then $([\varphi_2]+...+[\varphi_r])_{an}\equiv_{G(\varphi_1)}2^n\langle-1\rangle$. Then
$$\dim_{W,G}(\varphi)\ge\dim_{W,G}(\varphi_2+...+\varphi_r)
\ge\dim_{W,G}(\varphi_2+...+\varphi_r)_{an}\ge\dim_{W,G}([\varphi_2]+...+[\varphi_r])_{an}=2^n.$$

\item [II - ] $[\varphi]$ is not hyperbolic and $[\varphi_2]+...+[\varphi_r]$ is anisotropic. Then $\varphi_2+...+\varphi_r$ is anisotropic. By induction hypothesis we have $\dim_{W,G}(\varphi_2+...+\varphi_r)\ge2^n$. Then
$$\dim_{W,G}(\varphi)\ge\dim_{W,G}(\varphi_2+...+\varphi_r)\ge2^n.$$

\item [III - ] $[\varphi]$ is not hyperbolic and $[\varphi_2]+...+[\varphi_r]$ is isotropic. Since $[\varphi]$ is not hyperbolic, we can assume that $[\varphi_2]$ is anisotropic (otherwise, if $[\varphi_j]$ is isotropic for all $j=2,...,r$ then $[\varphi]$ is an isotropic Pfister form and then, is also hyperbolic). Denote $G_1:=G(\varphi_1)$. In $G_1([\varphi_2])$ (which is a reduced special group) look at 
$$\psi_2:=[[\varphi_2]+...+[\varphi_r]]=2^n\langle1\rangle+[\varphi_3]+...+[\varphi_r]\in G_1([\varphi_2]).$$
For $\psi_2\in I^n(G_1([\varphi_2]))$ we have
$$\dim_{W,G}(\varphi)\ge\dim_{W,G(\varphi_1)}[\varphi]\ge
\dim_{W,G_1([\varphi_2])}[\psi_2]$$
and the same cases I, II and III for $\psi_2$. Suppose without loss of generality that we are in case III, i.e, that $[\varphi_3]+...+[\varphi_r]$ is isotropic in $G_1([\varphi_2])$. If $[\varphi_j]$ is isotropic in $G_1([\varphi_2])$ for all $j\ge3$, then we are in case I. Now suppose $[\varphi_3]$ anisotropic in $G_1([\varphi_2])$ and denote $G_2:=G_1([\varphi_2])$. In $G_2([\varphi_3])$ (which is a reduced special group) look at 
$$\psi_3:=[[\varphi_3]+...+[\varphi_r]]=2^n\langle1\rangle+[\varphi_4]...+[\varphi_r]\in S_{F_2}([\varphi_3]).$$
For $\psi_3\in I^n(G_2([\varphi_3]))$ we have
$$\dim_{W,G}(\varphi)\ge\dim_{W,G(\varphi_1)}[\varphi]\ge
\dim_{W,G_1([\varphi_2])}[\psi_2]\ge
\dim_{W,G_2([\varphi_3])}[\psi_3].$$
and the same cases I, II and III for $\psi_3$. Repeating this process more $r-3$ times, we get at $[\varphi_r]$ in $G_{r-2}([\varphi_{r-1}])$ and
\begin{align*}
  \dim_{W,G}(\varphi)&\ge\dim_{W,G(\varphi_1)}[\varphi]\ge
\dim_{W,G_1([\varphi_2])}[\psi_2] \\
&\ge\dim_{W,G_2([\varphi_3])}[\psi_3] \ge...\ge
\dim_{W,G_{r-2}([\varphi_{r-1}])}[\psi_{r-1}].  
\end{align*}
Now, if $[\varphi_r]$ is isotropic in $G_{r-2}([\varphi_{r-1}])$ then $[\varphi_r]$ is hyperbolic in $G_{r-2}([\varphi_{r-1}])$, which by case I imply $\dim_{W,G_{r-2}([\varphi_{r-1}])}[\psi_{r-1}]\ge2^n$. If $[\varphi_r]$ is anisotropic in $G_{r-2}([\varphi_{r-1}])$  we are in case II and also $\dim_{W,G_{r-2}([\varphi_{r-1}])}[\psi_{r-1}]\ge2^n$.
\end{description}

Now suppose $-1\in D_G(\varphi)$. Then $G(\varphi_1)\cong\{1\}$, which imply $[\varphi]$ is hyperbolic, enabling us to use the very an adapted version argument in Case (I) above: the equivalence class of $\varphi$ on $G(\varphi_1)$ still is given by
$$[\varphi]=[\varphi_1+...+\varphi_r]=[\varphi_1]+...+[\varphi_r]=2^n\langle1\rangle+[\varphi_2]+...+[\varphi_r].$$
Then we have $[\varphi_2]+...+[\varphi_r]\equiv_{G(\varphi_1)}2^n\langle-1\rangle$, implying that
$$\dim_{W,G}(\varphi)\ge\dim_{W,G}(\varphi_2+...+\varphi_r)
\ge\dim_{W,G}([\varphi_2]+...+[\varphi_r])=2^n.$$
\end{proof}

The main difficult to extend the proof of \ref{haupred} for general special groups is the fact that $1-a$ is a saturated subgroup of $G$ iff $G$ is reduced, then the construction $G(a)$ is only available for reduced special groups.

Another tentative would be considering the special group of a field extension and try to mimic this structure in the general case. But it is highly non trivial, as we will see in the following Examples.

\begin{ex}[Quadratic Field Extensions and Quadratic Hyperfield Extensions]\label{exquad1}
Let $F$ be a formally real field and $p\in F^*$ such that $x^2-p$ has no roots in $F$. Consider $K=F(\sqrt p)$. Of course, we have two special hyperfields (and special groups) $G(K):=K/_m(K^2)^*$ and $G_{red}(K)=K/_m(\sum K^2)^*$. Note that if $a+b\sqrt p\in \dot K$, then
$$(a+b\sqrt p)^2=a^2+pb^2+2ab\sqrt p\in D_F(\langle1,p\rangle)+\sqrt p\cdot F,$$
where $D_F(\langle1,p\rangle)$ is the usual set of representatives of the $F$-quadratic form $\langle1,p\rangle$:
$$D_F(\langle1,p\rangle):=\{x^2+y^2p: x,y\in\dot F\}.$$
In other words,
$$K^2\setminus\{0\}\subseteq D_F(\langle1,p\rangle)+\sqrt p\cdot F.$$
Moreover,
\begin{align*}
   (a+b\sqrt p)^2+(c+d\sqrt p)^2
   &=(a^2+pb^2+2pab\sqrt p)+(c^2+pd^2+2cd\sqrt p) \\
   &=(a^2+pb^2+c^2+pd^2)+2(ab+cd)\sqrt p.
\end{align*}
Using the fact that $D_F(\langle1,p\rangle)\cdot D_F(\langle1,p\rangle)\subseteq D_F(\langle1,p\rangle)$ and for $a,b,c,d\ne0$,
$$a^2+pb^2+c^2+pd^2\in D_F(\langle1,p,1,p\rangle)=
D_F(\langle1,p\rangle\otimes\langle1,p\rangle)=D_F(\langle1,p\rangle)\cdot D_F(\langle1,p\rangle),$$
we conclude by induction (and a case analysis involving $0\in\{a,b,c,d\}$) that
$$\sum K^2\setminus\{0\}\subseteq D_F(\langle1,p\rangle)+\sqrt p\cdot F.$$
So, let $Q_p:=D_K(\langle1,p\rangle)+\sqrt p\cdot F$. Then $Q_p\cdot Q_p$ is a multiplicative set containing $\sum K^2$. Define $G_{\sqrt p}(K):=K/_mQ_p$. Then $G_{\sqrt p}(K)$ is a reduced special group such that
\begin{align*}
    \tag{*}G(K) \twoheadrightarrow G_{red}(K) \twoheadrightarrow G_{\sqrt p}(K).
    \end{align*}
Moreover,
$$G_{\sqrt p}(K)\cong S_{K/_m(K^2)^*}(\sqrt p),$$
i.e, the hyperfield of Theorem \ref{teo150}. We say that $K$ is \textbf{$p$-special} if $G_{\sqrt p}(K)\cong G_{red}(K)$. 
\end{ex}

\begin{ex}[The Special Group of a quadratic extension]\label{exquad2}
Let $F$ be a formally real field and $p\in F^*$ such that $x^2-p$ has no roots in $F$. Consider $K=F(\sqrt p)$. Using the calculations of Example \ref{exquad1} we have
\begin{align*}
    \dot K^2&=D_F(1,p)+\{x^2+yb^2+z\sqrt p:x,y\ne\dot F\mbox{ and }z=(x+y)^2-(x^2+y^2)\} \\
    &=\{x^2+yb^2+z\sqrt p:x,y\ne F\mbox{ are not both 0 and }z=(x+y)^2-(x^2+y^2)\}.
\end{align*}
In this sense, for $a,b,c,d\in\dot K$, what means $\langle a,b\rangle\equiv_K\langle c,d\rangle$ in terms of the isometry relation on $F$?

By Lemma 1.5(a) of \cite{dickmann2000special} we have
\begin{align*}
    \langle a,b\rangle\equiv_K\langle c,d\rangle\mbox{ iff }ab=cd\mbox{ and }
    ac\in D_K(1,cd).
\end{align*}
Lets first understand what means $\beta \in D_K(1,\alpha)$ for $\alpha,\beta \in\dot K$. By definition,
$$\beta \in D_K(1,\alpha)\mbox{ iff }\beta =x^2+\alpha y^2,\,x,y\in\dot K.$$
Write $\alpha=\alpha_1+\alpha_2\sqrt p$, $\beta =\beta_1+\beta_2\sqrt p$, $x=x_1+x_2\sqrt p$ and $y=y_1+y_2\sqrt p$. Then
\begin{align*}
    \beta&=x^2+\alpha y^2\Leftrightarrow\\
    \beta_1+\beta_2\sqrt p&=(x_1+x_2\sqrt p)^2+(\alpha_1+\alpha_2\sqrt p)(y_1+y_2\sqrt p)^2\Leftrightarrow \\
    \beta_1+\beta_2\sqrt p&=(x_1^2+px_2^2+2x_1x_2\sqrt p)+
    (\alpha_1+\alpha_2\sqrt p)(y_1^2+py_2^2+2y_1y_2\sqrt p)\Leftrightarrow \\
    \beta_1+\beta_2\sqrt p&=(x_1^2+px_2^2+\alpha_1y_1^2+\alpha_1py_2^2+2p\alpha_2y_1y_2)
    +(2x_1x_2+2\alpha_1y_1y_2+\alpha_2y_1^2+\alpha_2py_2^2)\sqrt p \\
    &\Leftrightarrow\begin{cases}\beta_1=x_1^2+px_2^2+\alpha_1y_1^2+\alpha_1py_2^2+2p\alpha_2y_1y_2
    \\\beta_2=2x_1x_2+2\alpha_1y_1y_2+\alpha_2y_1^2+\alpha_2py_2^2
    \end{cases} \\
    &\Leftrightarrow\begin{cases}\beta_1+\beta_2=(x_1+x_2)^2+(\alpha_1+\alpha_2p)(y_1+y_2)^2+(p-1)(x_2^2+\alpha_1y_2^2-\alpha_2y_1^2)
    \\\beta_1-\beta_2=(x_1-x_2)^2+(\alpha_1-\alpha_2p)(y_1-y_2)^2+(p-1)(x_2^2+\alpha_1y_2^2+\alpha_2y_1^2)
    \end{cases}
\end{align*}
Then
\begin{align}\label{eqkrep1}
    \beta&=x^2+ay^2\Leftrightarrow\begin{cases}\beta_1+\beta_2=(x_1+x_2)^2+(a_1+a_2p)(y_1+y_2)^2+(p-1)(x_2^2+a_1y_2^2-a_2y_1^2)
    \\\beta_1-\beta_2=(x_1-x_2)^2+(a_1-a_2p)(y_1-y_2)^2+(p-1)(x_2^2+a_1y_2^2+a_2y_1^2)
    \end{cases}
\end{align}

For the discriminant part, let $a,b,c,d\in\dot K$ with $a=a_1+a_2\sqrt p$, $b=b_1+b_2\sqrt p$, $c=c_1+c_2\sqrt p$ and $d=d_1+d_2\sqrt p$ for suitable $a_i,b_i,c_i,d_i\in F$ ($i=1,2$). We have
\begin{align*}
    ab=cd&\Leftrightarrow
    (a_1+a_2\sqrt p)(b_1+b_2\sqrt p)=(c_1+c_2\sqrt p)(d_1+d_2\sqrt p) \\
    &\Leftrightarrow(a_1b_1+pa_2b_2)+(a_1b_2+a_2b_1)\sqrt p=
    (c_1d_1+pc_2d_2)+(c_1d_2+c_2d_1)\sqrt p \\
    &\Leftrightarrow
    \begin{cases}a_1b_1+pa_2b_2=c_1d_1+pc_2d_2 \\ a_1b_2+a_2b_1=c_1d_2+c_2d_1\end{cases} \\
    &\Leftrightarrow
    \begin{cases}(a_1b_1+pa_2b_2)+(a_1b_2+a_2b_1)=(c_1d_1+pc_2d_2)+(c_1d_2+c_2d_1) \\ (a_1b_1+pa_2b_2)-(a_1b_2+a_2b_1)=(c_1d_1+pc_2d_2)-(c_1d_2+c_2d_1)\end{cases} \\
    &\Leftrightarrow
    \begin{cases}(a_1+a_2)(b_1+b_2)+(p-1)a_2b_2=(c_1+c_2)(d_1+d_2)+(p-1)c_2d_2 \\ 
    (a_1-a_2)(b_1-b_2)+(p-1)a_2b_2=(c_1-c_2)(d_1-d_2)+(p-1)c_2d_2\end{cases}
\end{align*}
Then
\begin{align}\label{eqkrep2}
    ab=cd&\Leftrightarrow
    \begin{cases}(a_1+a_2)(b_1+b_2)+(p-1)a_2b_2=(c_1+c_2)(d_1+d_2)+(p-1)c_2d_2 \\ 
    (a_1-a_2)(b_1-b_2)+(p-1)a_2b_2=(c_1-c_2)(d_1-d_2)+(p-1)c_2d_2\end{cases}
\end{align}

Now, since $ac=(a_1c_1+pa_2c_2)+(a_1c_2+a_2c_1)\sqrt p$ and
$ad=(a_1d_1+pa_2d_2)+(a_1c_2+a_2d_1)\sqrt p$, using Equations \ref{eqkrep1} and \ref{eqkrep2} (with $\beta=ac$ and $\alpha=ad$ in Equation \ref{eqkrep1}), we have the following characterization:
\begin{align*}
    &\langle a,b\rangle\equiv_K\langle c,d\rangle\mbox{ if and only if there exists }x_1,x_2,y_1,y_2\in F
    \mbox{ such that }(x_1,x_2),(y_1,y_2)\ne(0,0)\mbox{ and }\\
    &\begin{cases}
(a_1+a_2)(b_1+b_2)+(p-1)a_2b_2=(c_1+c_2)(d_1+d_2)+(p-1)c_2d_2 \\
    (a_1-a_2)(b_1-b_2)+(p-1)a_2b_2=(c_1-c_2)(d_1-d_2)+(p-1)c_2d_2 \\
    (a_1+a_2)(c_1+c_2)+(p-1)a_2c_2=(x_1+x_2)^2
    +[p(a_1+a_2)(d_1+d_2)-(p-1)a_1d_1](y_1+y_2)^2\\
    \qquad\qquad\qquad\qquad\qquad\qquad\qquad +(p-1)[x_2^2+(a_1d_1+pa_2d_2)y_2^2-(a_1c_2+a_2d_1)y_1^2] \\
    (a_1-a_2)(c_1-c_2)+(p-1)a_2c_2=(x_1-x_2)^2
    +[p(a_1-a_2)(d_1-d_2)-(p-1)a_1d_1](y_1-y_2)^2\\
    \qquad\qquad\qquad\qquad\qquad\qquad\qquad
    +(p-1)[x_2^2+(a_1d_1+pa_2d_2)y_2^2+(a_1c_2+a_2d_1)y_1^2]
\end{cases}
\end{align*}
Manipulating Equation \ref{eqkrep2} we get a very similar system to describe when $\overline a=\overline b$ in $K/_m\dot K^2$, $a,b\in K$.
\end{ex}

We avoid these complications considering more general multivalued structures: the superrings. In the realm of superrings, we built a general Marshall quotient and provide an adequate pair $(F,T)$ of a superfield and a multiplicative subset $T\subseteq F$ which will generalize $G(a)$.

\section{The Marshall Quotient of a Superfield}\label{polynomial-section}

In the realm of multirings, the notion of the so called "Marshall's  quotient", introduced in \cite{marshall2006real} and further developed in \cite{ribeiro2016functorial}, is a quotient multiring  defined for pair $(A,S)$ where $A$ is a multiring and $S \subseteq A$ is a multiplicative subset: given $a, b \in A $, 
$$a \approx_S b \ \mbox{ iff there are}\ x,y \in S \ \mbox{such that}\  ax =by.$$

Now we introduce the following:

\begin{defn}
Let $A$ be a superring and $S \subseteq A$. The set $S$ is called \textbf{Marshall coherent} if it is multiplicative ($1\in S$ and $S\cdot S\subseteq S$) and given $x,a \in A$ with $x \in as$ for some $s \in S$, there are $P,Q \subseteq S$ such that $xP = aQ$. We say that $S$ is \textbf{nontrivial Marshall coherent} if $0\notin S$.
\end{defn}

Let $A$ be a superring with $S \subseteq A$ Marshall coherent. For $a,b\in A$, define
$$a\sim_S b\mbox{ iff there are non-empty subsets }X,Y\subseteq S\mbox{ with }aX=bY.$$

\begin{fat}
If $A$ is a multiring viewed as a superring, then every multiplicative subset $S \subseteq A$ is Marshall coherent and  the above quotient notion coincides with the original Marshall quotient, i.e. $\approx_S = \sim_S$. 
\end{fat}

\begin{lem}
$ $
\begin{enumerate}[i -]
    \item For $a,b \in A$, the following are equivalent:
        \begin{enumerate}[a)]
            \item $a \sim b$.
            \item There exists $s,t \in S$ such that $as \cap bt \neq \emptyset$.
            \item There are $s,t,p,q \in S$ with 
            $a(st) = b(pq)$.
        \end{enumerate}
    \item The relation $\sim$ is an equivalence relation.
\end{enumerate}
\end{lem}
\begin{proof}
we only need to deal with the case $S$ nontrivial.
\begin{enumerate}[i -]
    \item The implication $c) \Rightarrow a)$ is straightforward. For $a) \Rightarrow b)$, let $X,Y \subseteq S$ such that $aX = bY$. Then there are $s \in X$ and $t \in Y$ such that $as \cap bt \neq \emptyset$. On the other hand, for $b) \Rightarrow c)$, let $x \in as \cap bt$. Thus, by Marshall coherence, there are $M,N,P,Q \subseteq S$ such that $xM = aP$ and $xN = bQ$. Therefore, $$a(PN) = x(MN) = b(QM).$$
    \vspace{-0.8cm}
    
    \item Let $a,b,c\in A$. 
    \begin{itemize}
    \item Since $a\cdot\{1\}=a\cdot\{1\}$ and $1\in S$, we have $a\sim a$.
    
    \item If $a\sim b$, then $aX=bY$ for some $X,Y \subseteq S$. So $bY=aX$ and $b\sim a$.
    
    \item Let $a\sim b$ and $b\sim c$. Then $aX=bY$ and $bZ=cW$ for some $X,Y,Z,W \subseteq S$. Hence
    $$a(XZ) = b(YZ) = c(WY)$$ and so $a \sim c$.
\end{itemize}
\end{enumerate}
\end{proof}

Now, let $A/_mS$ be the set of equivalence classes of $\sim $. We want to prescribe a superring structure for $A/_mS$. 

For $a\in A$, let $[a]$ be the equivalence class of $a$ in $A/_mS$. Define for $[a],[b]\in A/_mS$ the congruency relations:
\begin{align*}
    [c]&\in[a]+[b]\mbox{ iff there exist } c',a',b' \in A \mbox{ with } c' \in a' + b' \mbox{ and } c' \sim c, a' \sim a, b' \sim b.\\
    [c]&\in[a][b]\mbox{ iff there exist } c',a',b' \in A \mbox{ with } c' \in a'\cdot b' \mbox{ and } c' \sim c, a' \sim a, b' \sim b.\\
    [-a] & := -[a].
\end{align*}

\begin{lem}\label{lemsum1}
 Let $A$ be a superring and $S \subseteq A$ a Marshall coherent subset. Let $a,b,c \in A$.
 \begin{enumerate}[i -]
     \item $[c] \in [a] + [b]$ iff there is $s \in S$ such that $cs \subseteq aS + bS$.
     \item $[c] \in [a] \cdot [b]$ iff there is $s \in S$ such that $cs \subseteq abS$.
 \end{enumerate}
\end{lem}
\begin{proof}
 we only need to deal with the case $S$ nontrivial.
 \begin{enumerate}[i -]
    \item $(\Rightarrow)$: Let $c',a',b' \in A$ such that $c' \in a' + b'$ and $c' \sim c, a' \sim a, b' \sim b$. Then 
    $$c'X' = cX, a'Y' = aY, b'Z' = bZ\mbox{ for some }X,Y,Z,X',Y',Z' \subseteq S$$ 
    and so
    $$c(XY'Z') = c'(X'Y'Z') \in a'(X'Y'Z') + b'(X'Y'Z') = a(X'YZ') + b(X'Y'Z) \subseteq aS + bS.$$
    Therefore, for any $s \in XY'Z' \subseteq S$, we have $cs \subseteq aS + bS$.\\
    $(\Leftarrow)$: By hypothesis, there is $c' \in cs \cap at + bv$ for some $t,v \in S$. Therefore there exists $a' \in at$ and $b' \in bv$ with $c' \in a' + b'$. Lastly, Marshall coherence implies that $c' \sim c, a' \sim a$ and $b' \sim b$.
    
    \item $(\Rightarrow)$: Let $c',a',b' \in A$ such that $c' \in a' \cdot b'$ and $c' \sim c, a' \sim a, b' \sim b$. Then 
    $$c'X' = cX, a'Y' = aY, b'Z' = bZ\mbox{ for some }X,Y,Z,X',Y',Z' \subseteq S$$
    and so
    $$c(XY'Z') = c'(X'Y'Z') \in a'(X'Y'Z')b'(X'Y'Z') = a(X'YZ')b(X'Y'Z) \subseteq bcS.$$ Therefore, for any $s \in XY'Z' \subseteq S$, we have $cs \subseteq abS$.\\
 \end{enumerate}
\end{proof}

\begin{teo} \label{marshallquo-teo}
Let $A$ be a superring and $S\subseteq A$ a Marshall coherent subset. 

\begin{enumerate}[i -]
    \item The structure $(A/_mS,+,\cdot,-,[0],[1])$ is a superring.
    
    \item The projection map $\pi \colon A \to A /_m S$ is a universal morphism satisfying $\pi(S) = \{1\}$, that is, given a morphism $f \colon A \to B$ with $f(S) = \{1\}$, there is an unique morphism $\overline{f} \colon A/_m S \to B$ such that $f = \overline{f} \circ \pi$. In other words, for every morphism 
$f:A\rightarrow B$ such that $f[S]=\{1\}$, there exist a unique morphism $\tilde f:A/_mS\rightarrow B$ such 
that the following diagram commute:
$$\xymatrix{A\ar[r]^{\pi}\ar[dr]_{f} & A/_mS\ar[d]^{!\tilde f} \\ & B}$$
where $\pi:A\rightarrow A/_mS$ is the canonical projection $\pi(a)=\overline a$.
\end{enumerate}
\end{teo}

\begin{proof}
we only need to deal with the case $S$ nontrivial.
\begin{enumerate}[i -]
\item Firstly, we prove that $A/_m S$ is a superring.
\begin{itemize}
    \item $(A/_mS, +, -, [0])$ is a multigroup.\\
    The commutativity of $+$ is straightforward. Let $a,b,c \in A$.
    \begin{itemize}
        \item $[c] \in [a] + [b] \Rightarrow -[a] \in -[b] + [c]$.\\
        Let $c',a',b' \in A$ with $c' \in a' + b'$ and $c' \sim c, a' \sim a, b' \sim b$. Then $-a' \in -c' + b'$ and so $-[a] \in -[b] + [c]$.
        
        \item $[a] + [0] = \{[a]\}$.\\
        Let $[x] \in [a] + [0]$. Then there is $s \in S$ such that $xs \subseteq aS$. So, by Marshall coherence, $[x] = [a]$.
        
        \item $([a] + [b]) + [c] \subseteq [a] + ([b] + [c])$.\\
        Let $[x] \in [y] + [c]$, with $[y] \in [a] + [b]$. Then there are $s,t \in S$ such that $xs \subseteq yS + cS, yt \subseteq aS + bS$. Thus
        $$xst \subseteq (aS + bS) + cS \subseteq aS + (bS + cS).$$ It follows by Marshall coherence that $[x] \in [a] + [l], [l] \in [b] + [c].$
    \end{itemize}
    \item $(A/_mS, \cdot,[1])$ is a multimonoid.\\
    The comutativity of $\cdot$ is straightforward too. So let $a,b,c \in A$.
    \begin{itemize}
        \item $([a]\cdot [b]) \cdot [c] \subseteq [a]\cdot ([b] \cdot [c])$.\\
        Let $[x] \in [y]\cdot[c]$ with $[y] \in [a][b]$. Then there are $s,t \in S$ such that $xs \subseteq ycS, yt \subseteq abS$. Thus $$xst \subseteq (ab)cS \subseteq a(bc)S$$ and by Marshall coherence $[x] \in [a][l], [l] \in [b][c]$.
        \item $[a] \cdot[1] = \{[a]\}$.\\
        Let $[x] \in [a] \cdot [1]$. Then there is $s \in S$ with $xs \subseteq aS$. By Marshall coherence $[x] = [a]$.
    \end{itemize}
\end{itemize}

The verification of axioms $iii, iv$ and $v$ of Definition \ref{def:sup} are straightforward.

\item It is follows imediately from Marshall quotient definition that $\pi \colon A \to A/_m S$ is a morphism satisfying $\pi(S) = \{1\}$. Now let $f \colon A \to B$ be a morphism with $f(S) = \{1\}$. Note that if $a \sim b$, then $as = bt$ for some $s,t \in S$ and so for any $x \in as \cap bt$ we have $f(x) \in f(as) \cap f(bt) \subseteq \{f(a)\} \cap \{f(b)\}$. Thus $f(a) = f(b)$. Then we can define $\overline{f} \colon A/_m S \to B$ by $\overline{f}([a]) = f(a)$. Since the multioperations are define by congruency relations, we have that $\overline{f}$ is a superring morphism. 
\end{enumerate}
\end{proof}

As an immediate consequence, note that if $S, T$ are Marshall coherent subsets of $A$ and $S \subseteq T$, then we have a canonical surjective morphism of superrings:
$$A/_m S \twoheadrightarrow A/_m T.$$

\begin{cor}
Let $A$ be a superring and $S\subseteq A$ be a non-trivial Marshall coherent subset of $A$.
\begin{enumerate}[i -]
    \item If $A$ is full then $A/_mS$ is full.
    \item If $A$ is a superdomain then $A/_mS$ is a superdomain.
    \item If $A$ is a superfield then $A/_mS$ is a superfield.
\end{enumerate}
\end{cor}

\begin{teo}\label{qext1}
Let $A$ be a superdomain and $S\subseteq A\setminus\{0\}$ such that $1\in S$, $0\notin S$, $S\cdot S\subseteq S$ and $A^2\setminus\{0\}\subseteq S$. Then $S$ is a Marshall coherent subset of $S$. Moreover $A/_mS$ is a hyperfield, i.e, for all $[a],[b]\in A/_mS$, $[a][b]$ is a singleton  set.
\end{teo}
\begin{proof}
Let $a\in A$, $s\in S$ and $x\in as$. Suppose without loss of generality that $x\ne0$ (because if $x=0$ then $a=0$ because $A$ is a superdomain). Then $xa\subseteq a^2s\subseteq S$ and since
$$x(xa)=ax^2\mbox{ with }xa\mbox{ and }x^2\mbox{ contained in }S,$$
we have that $S$ is a Marshall coherent subset. Now let $[c],[d]\in[a].[b] \neq \emptyset$. Then $cs_1\subseteq abS$ and $ds_2\subseteq abS$ for some $s_1,s_2\in S$ (see Lemma \ref{lemsum1}). 

If $c=0$ or $d=0$ then $0\in abS$ which imply $a=0$ or $b=0$ and $[a][b]=\{[0]\}$. Let $c,d\ne0$. Then
$$cds_1s_2\subseteq a^2b^2S\cdot S\subseteq S.$$
Using this fact we get
$$c(cds_1s_2)=d(c^2s_1s_2)\mbox{ with }cds_1s_2\mbox{ and }c^2s_1s_2\mbox{ contained in }S.$$
Moreover $c\sim d$, thus $[a].[b]$ is a singleton.

We already know that $A/_mS$ is a superdomain. To show that $A/_mS$ is a superfield, it suffices to note that for each $a \in A$ such that $[a] \neq [0]$, $[1] \in [a].[a]$, or, equivalently, there is $s \in S$ such that $1s \subseteq aS.aS$, but $a^2 \in S$ and $1.a^2 \subseteq (a.1).(a.1)$. Thus $A/_mS$ is superfield with single-valued products, i.e., $A/_mS$ is an hyperfield.
\end{proof}

From Theorem \ref{qext1}, we have many examples of Marshall coherent sets for superdomains $A$ (of course, after removing zero):
\begin{itemize}
    \item the squares
    $$A^2:=\bigcup_{a\in A}a^2;$$
    \item the sum of squares
    $$\sum A^2:=\bigcup_{a_1,...,a_n\in A,\,n\in\mathbb N}a_1^2+a_2^2+...+a_n^2;$$
    \item preorderings, that are subsets $T\subseteq A$ with $T+T\subseteq T$, $T\cdot T\subseteq T$ and $A^2\subseteq T$.
\end{itemize}

\section{The Main Theorem: General Arason-Pfister Hauptsatz and  consequences} \label{Hauptsatz-section}

Throughout this Section we establish the following notation: Let $G$ be a special group and $F$ its special hyperfield associated. In particular, if $\alpha\in F$, $\alpha\ne0,1$, the polynomial $p(X)\in F[X]$, $p(X)=X^2-\alpha$ has no roots in $F$ (basically because $\alpha^2=1$ for all $\alpha\in F^*$).

Then, $F(p)\cong F[X]/\langle X^2-\alpha\rangle$ is a hyperfield, and we denote its elements by $[f(X)]$, $f(X)\in F[X]$. Note that $F(p)$ is generated by $[X]$, and denoting $\omega=[X]$, we have $0\in\omega^2-\alpha$ and
$$F(p)=\{[a]+[b]\omega:a,b\in F\}.$$

We Define the following:
\begin{defn}\label{sgalgext}
    Let $G$ be a special group, $F$ its special hyperfield associated and $\alpha\in F$ with $\alpha\ne0,1$. We denote $\omega=[X]\in F[X]/\langle X^2-\alpha\rangle$ and define
    \begin{align*}
        F(\omega)&:=F[X]/\langle X^2-\alpha\rangle;\\
        S_F(\omega)&=(F(\omega))/_m(F(\omega)^2\setminus\{0\});\\
        S^{red}_F(\omega)&=(F(\omega))/_m\left(\sum F(\omega)^2\setminus\{0\}\right).
    \end{align*}
    We also write $\omega=\sqrt\alpha$.
\end{defn} 

\begin{prop}\label{prop1}
 Let $F$ be a formally real special hyperfield and $S_F(\omega)$ as above.
 \begin{enumerate}[a -]
  \item Let $a_1,a_2,...,a_n,b_1,b_2,...,b_n\in F$. Then
 $$(a_1+b_1\omega)^2+(a_2+b_2\omega)^2+...+(a_n+b_n\omega)^2\subseteq(1+\alpha)+2[a_1b_1+a_2b_2+...+a_nb_n]\omega,$$
 where $2X:=X+X$.
  \item $\sum F(\omega)^2\setminus\{0\}=(1+\alpha)+F\cdot\omega$.
  \item Denote $(2)F=\bigcup\{x+x:x\in F\}$. Then
  $$F(\omega)^2=(1+\alpha)+(2)F\cdot\omega.$$
  \item $F(\omega)^2=\sum F(\omega)^2$ iff $(2)F=F$.
  \item $-1\notin\sum F(\omega)^2$ iff $-1\notin1+\alpha$. 
  \item $\omega\notin F(\omega)^2$.
  \item The morphism $F\rightarrow S_F(\omega)$ is full and not injective.
  \item If $-1\notin1+\alpha$ and $a,b\in a+b$ for all $a,b\in \dot F$ then $S_F(\omega)$ is a real reduced hyperfield (and then, a reduced special group).
 \end{enumerate}
\end{prop}
\begin{proof}
$ $
\begin{enumerate}[a -]
 \item In fact, if $n=2$, then
 \begin{align*}
  (a_1+b_1\omega)^2+(a_2+b_2\omega)^2&=a_1^2+a_1b_1\omega+a_1b_1\omega+b_1^2\omega^2+
  a_2^2+a_2b_2\omega+a_2b_2\omega+b_2^2\omega^2 \\
  &=1+[a_1b_1+a_1b_1]\omega+\alpha+1+[a_2b_2+a_2b_2]\omega+\alpha \\
  &=[1+\alpha+1+\alpha]+[a_1b_1+a_1b_1+a_2b_2+a_2b_2]\omega \\
  &\subseteq(1+\alpha)+2[a_1b_1+a_2b_2]\omega.
 \end{align*}
 Here, we use the fact $1+\alpha+1+\alpha=(1+\alpha)(1+\alpha)\subseteq1+\alpha$.
 
Now, suppose true for $n$ and let $a_1,a_2,...,a_n,a_{n+1},b_1,b_2,...,b_n,b_{n+1}\in F$.
\begin{align*}
 &(a_1+b_1\omega)^2+(a_2+b_2\omega)^2+...+(a_n+b_n\omega)^2+(a_{n+1}+b_{n+1}\omega)^2 \\
 &=[(a_1+b_1\omega)^2+(a_2+b_2\omega)^2+...+(a_n+b_n\omega)^2]+(a_{n+1}+b_{n+1}\omega)^2 \\
 &=(1+\alpha)+2[a_1b_1+...+a_nb_n]\omega+1+2a_{n+1}b_{n+1}\omega+\alpha \\
 &=(1+\alpha)+2[a_1b_1+...+a_nb_n+a_{n+1}b_{n+1}]\omega,
\end{align*}
as desired.
 
 \item Using the previous item, we get
 \begin{align*}
  (1-\omega)^2+(1+\omega)^2&=(1+\alpha)+2[1-1]\omega=(1+\alpha)+F\cdot\omega.
 \end{align*}
 Moreover, $(1+\alpha)+F\cdot\omega\subseteq\sum S_F(\omega)^2\setminus\{0\}$, completing the proof.
 
 \item Let $a,b\in F$. Then
 $$(a+b\omega)^2=a^2+ab\omega+ab\omega+b^2\omega^2=(1+\alpha)+2ab\omega.$$
 Then $F(\omega)^2\subseteq(1+\alpha)+(2)F\cdot\omega$. Conversely, let $t\in(1+\alpha)+(2)F\cdot\omega$. Then 
$t\in(1+\alpha)+2a\omega$ for some $a\in F$. Since
$$(1+\alpha)+2a\omega=(1+a\omega)^2,$$
we get $t\in(1+a\omega)^2\subseteq F(\omega)^2$, completing the proof.

\item Just use items (a), (b) and (c).

\item If $-1\in F(\omega)^2$, then $-1\in x+2z\omega$ for some $x\in1+\alpha$ and $z\in F$. If $z=0$ then $-1=1+\alpha$, 
contradiction. If $z\ne0$, then
$$0\in1-1\subseteq1+x+2z\omega\subseteq1+1+\alpha+2z\omega.$$
Then $0\in y+z\omega$ for some $y\in1+1+\alpha$, and then, $0\in ev(g(X),\omega)$, for $g(X)=y+zX$, contradicting the fact 
that $f(X)=X^2-\alpha$ is the minimum degree polynomial in $F(\omega)$ from which $0\in f(\omega)$.
 
 \item Just use the same argument of item (d).
 
 \item Let $t\in a+b$ for some $t\in x+2y\omega$, with $y\ne0$. Then
 $$-b\in a-t\subseteq a-x-2y\omega,$$
 and then,
 $$0\in b-b\subseteq b+a-x-2y\omega.$$
 Therefore exists some $d\in b+a-x$ such that $0\in d-2y\omega$, which imply that $0\in ev(g(X),\omega)$ for $g(X)=d-zX$ 
for some $z\in2y$ with $z\ne0$, contradiction. This morphism is not injective because if $a\in1+\alpha$ then $[a]=[1]$ in $S_F(\omega)$.

\item If $a,b\in a+b$ for all $a,b\in\dot F$, then $(2)F=F$. Hence $F(\omega)^2=\sum F(\omega)^2$, which imply $S_F(\omega)$ real reduced ($1+1=\{1\}$).
\end{enumerate}
\end{proof}

\begin{defn}[Rooted Superfield]
A superfield $F$ is \textbf{rooted} if 
$$\{a,b\}\subseteq a+b\mbox{ for all }a,b\in F\setminus\{0\}.$$
\end{defn}

\begin{teo}\label{teo150}
Let $F$ be a formally real pre-special hyperfield and $\omega\in\overline{F}\setminus F$ be a root of $f(X)=X^2-\alpha\in F[X]$. Suppose that $-1\notin1+\alpha$. Then $S_F(\omega)$ is a formally real pre-special hyperfield. Moreover if $F$ is rooted then $S_F(\omega)=S^{red}_F(\omega)$, and in particular, $S_F(\omega)$ is a real reduced hyperfield.
\end{teo}
\begin{proof}
We already know (using Theorem \ref{qext1} and item (e) of Proposition \ref{prop1}) that $S_F(\omega)$ is a formally real pre-special hyperfield.
If $F$ is rooted, then by item (h) of Proposition \ref{prop1} we have the desired.
\end{proof}

In the sequel, we want to iterate de construction $S_F(\omega)$. Let $\alpha,\beta\in \dot F\setminus\{-1\}$. The properties of Proposition \ref{prop1} are valid if we change $F$ by $S_F(\sqrt{\alpha})(\sqrt{\beta})$ (or $S_F(\sqrt{\beta})(\sqrt{\alpha})$).

\begin{teo}\label{sext1}
Let $F$ be a special hyperfield and $\alpha,\beta\in\dot F\setminus\{\pm1\}$. Then
$$S_{S_F(\sqrt{\alpha})}(\sqrt{\beta})\cong S_{S_F(\sqrt{\beta})}(\sqrt{\alpha}).$$
\end{teo}
\begin{proof}
We already know that 
$$F(\sqrt{\alpha})(\sqrt{\beta})\cong F(\sqrt{\beta})(\sqrt{\alpha})$$
and 
$$F(\sqrt{\alpha})(\sqrt{\beta})=F+F\sqrt{\alpha}+F\sqrt{\beta}+F\sqrt{\alpha}\sqrt{\beta}.$$
Let $\varphi:F(\sqrt{\alpha})(\sqrt{\beta})\rightarrow F(\sqrt{\beta})(\sqrt{\alpha})$ be an isomorphism and write $q_1:F(\sqrt{\alpha})(\sqrt{\beta})\rightarrow S_F(\sqrt{\alpha})(\sqrt{\beta})$ and $q_2:F(\sqrt{\beta})(\sqrt{\alpha})\rightarrow S_F(\sqrt{\beta})(\sqrt{\alpha})$ the natural projections. For instance, $q_1$ is given by the rule
$$q_1(a_0+a_1\sqrt{\alpha}+a_2\sqrt{\beta}+a_3\sqrt{\alpha}\sqrt{\beta}):=
[a_0+a_1\sqrt{\alpha}]+[a_2+a_3\sqrt{\alpha}]\sqrt{\beta}\in S_F(\sqrt{\alpha})(\sqrt{\beta}).$$
Similarly for $q_2$. 

Now, let $\pi_1:S_F(\sqrt{\alpha})(\sqrt{\beta})\rightarrow
S_{S_F(\sqrt{\alpha})}(\sqrt{\beta})$ and $\pi_2:S_F(\sqrt{\beta})(\sqrt{\alpha})\rightarrow
S_{S_F(\sqrt{\beta})}(\sqrt{\alpha})$ be the quotient morphisms. Since $\varphi[F(\sqrt{\alpha})(\sqrt{\beta})^2]=F(\sqrt{\beta})(\sqrt{\alpha})^2$ 
and $q_2[F(\sqrt{\beta})(\sqrt{\alpha})^2]\subseteq(S_F(\sqrt{\beta})(\sqrt{\alpha}))^2$,
we have that $\pi_2\circ q_2\circ\varphi:F(\sqrt{\alpha})(\sqrt{\beta})\rightarrow
S_{S_F(\sqrt{\beta})}(\sqrt{\alpha})$ is a morphism such that
$$\pi_2\circ q_2\circ\varphi[F(\sqrt{\alpha})(\sqrt{\beta})]=\{1\}.$$
By the universal property there is an unique morphism 
$\varphi_{\alpha\beta}:S_{S_F(\sqrt{\alpha})}(\sqrt{\beta})\rightarrow
S_{S_F(\sqrt{\beta})}(\sqrt{\alpha})$. Using the same argument, there is an unique morphism $\varphi_{\beta\alpha}:S_{S_F(\sqrt{\beta})}(\sqrt{\alpha})\rightarrow
S_{S_F(\sqrt{\alpha})}(\sqrt{\beta})$. The universal property forces $\varphi_{\alpha\beta}\circ\varphi_{\beta\alpha}=id$ and 
$\varphi_{\beta\alpha}\circ\varphi_{\alpha\beta}=id$.
\end{proof}

With Theorem \ref{sext1} we are able to properly iterate the construction $S_F(\omega)$. For $a_1,...,a_n\in\dot F$, we define recursively:
\begin{align*}
    S_{F(\sqrt{a_1},\sqrt{a_2})}&:=S_{S_F(\sqrt{a_1})}(\sqrt{a_2}); \\
    S_{F(\sqrt{a_1},...,\sqrt{a_{n+1}})}&:=
    S_{S_F(\sqrt{a_1},...,\sqrt{a_n})}(\sqrt{a_{n+1}}).
\end{align*}

\begin{cor}\label{cor333}
 Let $F$ be a special hyperfield, $\alpha_1,\alpha_2,...,\alpha_n\in\dot F$, and $\sigma\in S_n$. Then
 $$S_{F(\sqrt{a_1},...,\sqrt{a_n})}\cong
 S_{F(\sqrt{a_{\sigma(1)}},...,\sqrt{a_{\sigma(n)}})}.$$
\end{cor}


The next step is to describe the isometry relation $\equiv_{S_F(\omega)}$ in $S_F(\omega)$ in terms of isometry relation $\equiv_F$ in $F$.

If $\varphi=\langle a_1,...,a_n\rangle$ is a form on $F$, we denote the form $[\varphi]$ on $S_F(\omega)$ simply by $[\varphi]:=\langle [a_1],...,[a_n]\rangle$. We say that \textbf{$[\varphi]$ is the equivalence class of $\varphi$ in $S_F(\omega)$}. Of course, if $\varphi=\varphi_1\oplus\varphi_2$ (or $\varphi=\varphi_1\otimes\varphi_2$) then $[\varphi]=[\varphi_1]\oplus[\varphi_2]$ (or $[\varphi]=[\varphi_1]\otimes[\varphi_2]$). We have the following useful consequences of Proposition \ref{prop1} (which we will use freely): 
\begin{rem}
$ $
\begin{enumerate}[a -]
    \item If $\varphi\equiv\psi$ on $F$ then $[\varphi]\equiv[\psi]$ on $S_F(\omega)$;
    \item if $\varphi$ is isotropic on $F$ then $[\varphi]$ is isotropic on $S_F(\omega)$;
    \item if $[\varphi]$ is anisotropic on $S_F(\omega)$ then if $\varphi$ is anisotropic on $F$.
\end{enumerate}
\end{rem}

\begin{teo}\label{iso0}
 Let $a,b\in F$. Then $[a]=[b]$ in $S_F(\omega)$ iff there is $s,t\in1+\alpha$ with $as=bt$ (or $a=bst$).
\end{teo}
\begin{proof}
 ($\Rightarrow$) Suppose that $[a]=[b]$ in $S_F(\omega)$. Then there exist $X,Y\subseteq F(\omega)^2\setminus\{0\}$ with $aX=bY$. Let $r_1+s_1\omega\in X$, with $r\in1+\alpha$. Then 
 $$a(r_1+s_1\omega)=ar_1+as_1\omega\in bY,$$
 and there exist $r_2+s_2\omega\in Y$ with
 $$\emptyset\ne(ar_1+as_1\omega)\cap(b(r_2+s_2\omega)).$$
 But $b(r_2+s_2\omega)=\{br_2+bs_2\omega\}$. Then $ar_1+as_1\omega=br_2+bs_2\omega$, which imply $ar_1=br_2$ and $as_1=bs_2$.
 
 ($\Leftarrow$) Immediate since $1+\alpha\subseteq F(\omega)^2\setminus\{0\}$.
\end{proof}

\begin{teo}\label{teo32}
 Let $a,b\in F$. Then $[a]=[b]$ in $S_{F(\sqrt{a_1},...,\sqrt{a_n})}$ iff there is $s,t\in D_F(\langle\langle\alpha_1,\alpha_2,...,\alpha_n\rangle\rangle)$
 \footnote{Here $\langle\langle\alpha_1,\alpha_2,...,\alpha_n\rangle\rangle$ denotes the Pfister form $\langle1,\alpha_1\rangle\otimes\langle1,\alpha_2\rangle\otimes...\otimes\langle1,\alpha_n\rangle$.}
 such that $as=bt$ (or $a=bst$ or even $ab\in D_F(\langle\langle\alpha_1,\alpha_2,...,\alpha_n\rangle\rangle)$).
\end{teo}
\begin{proof}
Since $S_{F(\sqrt{a_1},...,\sqrt{a_n})}$ is a real reduced hyperfield and in $S_{F(\sqrt{a_1},...,\sqrt{a_n})}$, $\{[\alpha_1],...,[\alpha_n]\}=\{[1]\}$ we only need to prove $(\Rightarrow)$.

Let $a,b\in F$. If $[a]=[b]$ in $S_{F(\sqrt{\alpha_1},\sqrt{\alpha_2})}=S_{S_{F(\sqrt{\alpha_1})}}(\sqrt{\alpha_2})$, then by Theorem \ref{iso0} (changing $F$ by $S_{F(\sqrt{\alpha_1})}$) we have $[ab]\in[1]+[\alpha_2]$ (in $S_{F(\sqrt{\alpha_1})}$). Then there exist $s\in1+\alpha_1$ such that 
$$abs\in(1+\alpha_1)+\alpha_2(1+\alpha_1)=D_F(\langle1,\alpha_1,\alpha_2,\alpha_1\alpha_2\rangle)=D_F(\langle\langle\alpha_1,\alpha_2\rangle\rangle).$$
This means $abs\in D_F(\langle\langle\alpha_1,\alpha_2\rangle\rangle)$. 

Now suppose the desired valid for $n$. By induction hypothesis $[a]=[b]$ in 
$$S_{F(\sqrt{a_1},...,\sqrt{a_{n+1}})}\cong S_{(S_F(\sqrt{a_1}))(\sqrt{a_2},...,\sqrt{a_{n+1}})}\footnote{We are doing a convenient use of Corollary \ref{cor333}.}$$ iff
$$[ab]\in D_{S_F(\sqrt{a_1})}(\langle\langle[\alpha_2],
[\alpha_2],...,[\alpha_n]\rangle\rangle).$$
This imply
$$ab\in D_{F}(\langle\langle\alpha_2,
\alpha_2,...,\alpha_n\rangle\rangle)\cdot(1+\alpha_1)\subseteq
D_F(\langle\langle\alpha_1,\alpha_2,...,\alpha_{n+1}\rangle\rangle).$$
\end{proof}

\begin{cor}\label{cor33}
  Let $F$ be a special hyperfield and $\alpha_1,\alpha_2,...,\alpha_n\in\dot F\setminus\{\pm1\}$. Then $S_{F(\sqrt{a_1},...,\sqrt{a_n})}$ is formally real iff $-1\notin D_F(\langle\langle\alpha_1,\alpha_2,...,\alpha_n\rangle\rangle)$.
\end{cor}
\begin{proof}
Since $S_{F(\sqrt{a_1},...,\sqrt{a_n})}$ is a real reduced hyperfield, we have $S_{F(\sqrt{a_1},...,\sqrt{a_n})}$ formally real iff $[1]\ne[-1]$, which by Theorem \ref{teo32} occurs iff $-1\notin D(\langle\langle\alpha_1,\alpha_2,...,\alpha_n\rangle\rangle)$.
\end{proof}

\begin{teo}\label{iso2}
 Let $a,b,c,d\in\dot F$. Then $\langle[a],[b]\rangle\equiv_{S_F(\omega)}\langle[c],[d]\rangle$ iff
 $\langle ar,bs\rangle\equiv_{F}\langle c,dt\rangle$ for some $r,s,t\in1+\alpha$.
\end{teo}
\begin{proof}
  ($\Rightarrow$) Let $\langle[a],[b]\rangle\equiv_{S_F(\omega)}\langle[c],[d]\rangle$. Then $[a][b]=[c][d]$ and $[a][c]\in1+[c][d]$. Hence, there are $v,w\in1+\alpha$ and $x\in S=F(\omega)^2\setminus\{0\}$ with $abv=cdw$ (or $abvw=cd$) and $acx\in S+cdS$. Write $x=x_1+x_2\omega$. We have
  $$avwcx_1x\in vx_1(S+cdS)\subseteq vx_1S+cdvx_1S\subseteq S+cdS.$$
  Then 
 $$avwcx_1(x_1+x_2\omega)\in(y_1+y_2\omega)+cd(z_1+z_2\omega)$$
 for some $y_1,z_1\in1+\alpha$ and $y_2,z_2\in F$. This means
 $$avwc+avwcx_1x_2\omega\in(y_1+cdz_1)+(y_2+cdz_2)\omega,$$
 and then, $avwc\in y_1+cdz_1$ and $avwcx_1x_2\in y_2+cdz_2$. Then $avwcy_1\in1+cdy_1z_1$ or equivalently, $(avwy_1)c\in1+c(dy_1z_1)$. Therefore $(avwy_1)(bz_1)=c(dy_1z_1)$ with $(avwy_1)c\in1+c(dy_1z_1)$, which means $\langle avwy_1,bz_1\rangle\equiv_F\langle c,dy_1z_1\rangle$. Putting $r=vwy_1$, $s=z_1$ and $t=y_1z_1$ we get the desired.
 
 ($\Leftarrow$) Immediate.
\end{proof}

 Then for all $a,b,c,d\in\dot F$ are equivalent:
 \begin{enumerate}[i-]
     \item $\langle[a],[b]\rangle\equiv_{S_F(\omega)}\langle[c],[d]\rangle$;
     \item $\langle ar,bs\rangle\equiv_{F}\langle ct,d\rangle$ for some $r,s,t\in1+\alpha$.
     \item $\langle ar,bs\rangle\equiv_{F}\langle c,dt\rangle$ for some $r,s,t\in1+\alpha$.
     \item $\langle a,br\rangle\equiv_{F}\langle cs,dt\rangle$ for some $r,s,t\in1+\alpha$.
     \item $\langle ar,b\rangle\equiv_{F}\langle cs,dt\rangle$ for some $r,s,t\in1+\alpha$.
 \end{enumerate}

Now we have all the necessary to expand the validity of the Arason-Pfister Hauptsatz for all special groups. We use the same notations in \ref{pfisternotation}. The proof proceed similar to the one in Theorem \ref{haupred}.

\begin{teo}[Arason-Pfister Hauptsatz]\label{haup}
 Let $F$ be a special hyperfield, then  it holds $AP_F(n)$, for all $n \geq 0$. In more details: for each  $n \geq 0$ and For each $\varphi = \langle a_1,\cdots, a_k \rangle$, a  non-empty ($k\geq 1$), regular ($a_i \in \dot{F}$) and anisotropic form, if  $\varphi\in I^n(F)$, then $\dim(\varphi)\ge2^n$  $\varphi\in I^n(F) $, if $\varphi \neq \emptyset$ is anisotropic, then $\dim_{W,F}(\varphi)\ge2^n$.
\end{teo}

\begin{proof}
An equivalent way to state this result is the following: if a form $q$ belongs to $I^nF$ and $\dim q<2^n$ then $q$ must be a hyperbolic form.

 Since $\varphi$ is an anisotropic form such that $\varphi \in I^nF \setminus \{\emptyset\}$ and  $I^nF$ is additively generated by the Pfister forms, then there exists $r \geq 1$ and Pfister's forms of degree $n$, $\varphi_1, \cdots \varphi_r$ such that
$\varphi= \pm(\varphi_1+...+\varphi_r)$.  

Since $\varphi$ is anisotropic, we can suppose without loss of generality that $\varphi= \varphi_1+...+\varphi_r$ and proceed by induction on $r$. 

If $r=1$, then $\varphi=\varphi_1$, with $\dim(\varphi)=\dim(\varphi_1)=2^n$.

Let $r\ge2$. If $\varphi_j$ is isotropic for all $j=1,...,r$ then $\varphi$ is isotropic (hyperbolic, in fact): this fallows from Witt's cancellation law since $\varphi \oplus k \langle 1, -1 \rangle \equiv (r2^{n-1}+m) \langle 1, -1 \rangle $. So we can suppose without loss of generality that $\varphi_1=\langle\langle a_1,...,a_n\rangle\rangle$ is anisotropic.

Suppose $-1\notin D_F(\varphi_1)$. By Corollary \ref{cor33} Let $S_F(\varphi_1):=S_{F(\sqrt{a_1},...,\sqrt{a_n})}$. Then equivalence class of $\varphi$ on $S_F(\varphi_1)$ is
$$[\varphi]=[\varphi_1+...+\varphi_r]=[\varphi_1]+...+[\varphi_r]=2^n\langle1\rangle+[\varphi_2]+...+[\varphi_r].$$
We already know that $\dim_{W,F}(\varphi)\ge\dim_{W,S_F(\varphi_1)}[\varphi]$. Then we have three cases:
\begin{description}
\item [I - ] $[\varphi]$ is hyperbolic. Then $([\varphi_2]+...+[\varphi_r])_{an}\equiv_{S_F(\varphi_1)}2^n\langle-1\rangle$. Then
$$\dim_{W,F}(\varphi)\ge\dim_{W,F}(\varphi_2+...+\varphi_r)
\ge\dim_{W,F}(\varphi_2+...+\varphi_r)_{an}\ge\dim_{W,F}([\varphi_2]+...+[\varphi_r])_{an}=2^n.$$

\item [II - ] $[\varphi]$ is not hyperbolic and $[\varphi_2]+...+[\varphi_r]$ is anisotropic. Then $\varphi_2+...+\varphi_r$ is anisotropic. By induction hypothesis we have $\dim_{W,F}(\varphi_2+...+\varphi_r)\ge2^n$. Then
$$\dim_{W,F}(\varphi)\ge\dim_{W,F}(\varphi_2+...+\varphi_r)\ge2^n.$$

\item [III - ] $[\varphi]$ is not hyperbolic and $[\varphi_2]+...+[\varphi_r]$ is isotropic. 

Since $[\varphi]$ is not hyperbolic, we can assume that $[\varphi_2]$ is anisotropic (otherwise, if $[\varphi_j]$ is isotropic for all $j=2,...,r$ then $[\varphi]$ is an isotropic Pfister form and then, is also hyperbolic). write $F_1:=S_F(\varphi_1)$. In $S_{F_1}([\varphi_2])$ (which is a special hyperfield) look at 
$$\psi_2:=[[\varphi_2]+...+[\varphi_r]]=2^n\langle1\rangle+[\varphi_3]+...+[\varphi_r]\in S_{F_1}([\varphi_2]).$$
For $\psi_2\in I^n(S_{F_1}([\varphi_2]))$ we have
$$\dim_{W,F}(\varphi)\ge\dim_{W,S_F(\varphi_1)}[\varphi]\ge
\dim_{W,S_{F_1}([\varphi_2])}[\psi_2]$$
and the same cases I, II and III for $\psi_2$. Suppose without loss of generality that we are in case III, i.e, that $[\varphi_3]+...+[\varphi_r]$ is isotropic in $S_{F_1}([\varphi_2])$. If $[\varphi_j]$ is isotropic in $S_{F_1}([\varphi_2])$ for all $j\ge3$, then we are in case I. Now suppose $[\varphi_3]$ anisotropic in $S_{F_1}([\varphi_2])$ and write $F_2:=S_{F_1}([\varphi_2])$. In $S_{F_2}([\varphi_3])$ (which is a special hyperfield) look at 
$$\psi_3:=[[\varphi_3]+...+[\varphi_r]]=2^n\langle1\rangle+[\varphi_4]...+[\varphi_r]\in S_{F_2}([\varphi_3]).$$
For $\psi_3\in I^n(S_{F_2}([\varphi_3]))$ we have
$$\dim_{W,F}(\varphi)\ge\dim_{W,S_F(\varphi_1)}[\varphi]\ge
\dim_{W,S_{F_1}([\varphi_2])}[\psi_2]\ge
\dim_{W,S_{F_2}([\varphi_3])}[\psi_3].$$
and the same cases I, II and III for $\psi_3$. Repeating this process more $r-3$ times, we get at $[\varphi_r]$ in $S_{F_{r-1}}([\varphi_{r-1}])$ and
\begin{align*}
  \dim_{W,F}(\varphi)&\ge\dim_{W,S_F(\varphi_1)}[\varphi]\ge
\dim_{W,S_{F_1}([\varphi_2])}[\psi_2] \\
&\ge\dim_{W,S_{F_2}([\varphi_3])}[\psi_3] \ge...\ge
\dim_{W,S_{F_{r-2}}([\varphi_{r-2}])}[\psi_{r-1}].  
\end{align*}
Now, if $[\varphi_r]$ is isotropic in $S_{F_{r-1}}([\varphi_{r-1}])$ then $[\varphi_r]$ is hyperbolic in $S_{F_{r-1}}([\varphi_{r-1}])$, which by case I imply $\dim_{W,S_{F_{r-2}}([\varphi_{r-2}])}[\psi_{r-1}]\ge2^n$. If $[\varphi_r]$ is anisotropic in $S_{F_{r-1}}([\varphi_{r-1}])$  we are in case II and also $\dim_{W,S_{F_{r-2}}([\varphi_{r-2}])}[\psi_{r-1}]\ge2^n$.
\end{description}

Now suppose $-1\in D_F(\varphi)$. Then $S_F(\varphi_1)\cong\{0,1\}$ (see Theorem \ref{pre-krasner}), which imply $[\varphi]$ is hyperbolic, enabling us to use the very an adapted version argument in Case (I) above: the equivalence class of $\varphi$ on $S_F(\varphi_1)$ still is given by
$$[\varphi]=[\varphi_1+...+\varphi_r]=[\varphi_1]+...+[\varphi_r]=2^n\langle1\rangle+[\varphi_2]+...+[\varphi_r].$$
Then we have $[\varphi_2]+...+[\varphi_r]\equiv_{S_F(\varphi_1)}2^n\langle-1\rangle$, implying that
$$\dim_{W,F}(\varphi)\ge\dim_{W,F}(\varphi_2+...+\varphi_r)
\ge\dim_{W,F}([\varphi_2]+...+[\varphi_r])=2^n.$$
\end{proof}
 
$ $

Now, we turn our attention to graded rings associated to abstract quadratic forms theories (special hyperfields, or equivalently, special groups): we will apply the above established Theorem APH  to obtain information on the {\em inductive graded rings} (Definition 3.1 in \cite{dickmann1998quadratic}) of a special group $G$: the graded  Witt ring of $G$, 
$$W_*(G) = (I^n(G)/I^{n+1}(G) \overset{\langle 1, 1\rangle \otimes -}\longrightarrow I^{n+1}(G)/I^{n+2}(G))_{n \in \mathbb{N}},$$
and on  the graded ring of $k$-theory of $G$,  
$$k_*(G) = (k_n(G) \overset{\lambda(-1) \otimes -}\longrightarrow k_{n+1}(G))_{n \in \mathbb{N}}.$$

The uses of k-theoretic (and Boolean) methods in abstract theories of quadratic forms has been proved a very successful method, see for 
instance, these two papers of Dickmann and Miraglia: \cite{dickmann1998quadratic} where they give an affirmative answer to Marshall Signature Conjecture, and \cite{dickmann2003lam}, where they give an affirmative answer to Lam's Conjecture (previously both conjecture have kept open for almost three decades).  These two central papers makes us take a deeper look at the theory of special groups (and hence, hyperbolic/pre-special hyperfields) by itself. This is not mere exercise in abstraction: from Marshall's and Lam's Conjecture many questions arise in the abstract and concrete context of quadratic forms. 

We will freely permute between a special group $G$ and a special hyperfield $F$ since the associations $G \mapsto F_G : = G \dot{\cup} \{0\}$ and $F \mapsto G_F := F \setminus \{0\}$  are  part of an equivalence of categories (\cite{ribeiro2016functorial}, \cite{roberto2021quadratic}). The graded Witt ring of a special group is studied in \cite{dickmann1998quadratic} and \cite{dickmann2000special}; \cite{dickmann2006algebraic} is the reference for the k-theory of special groups; in \cite{roberto2021ktheory} is developed a k-theory for all hyperbolic hyperfields (that includes all pre-special hyperfields).

For the reader's convenience we recall below some relevant Definitions.

\begin{defn}[The Dickmann-Miraglia k-theory \cite{dickmann2006algebraic}]\label{defn:ksg}
 For each special group $G$ (written multiplicatively) we associate a (inductive) graded ring
 $$k_*G=(k_0G,k_1G,...,k_nG,...)$$
 as follows: $k_0G:=\mathbb F_2$ and $k_1G:=G$ written additively. With this purpose, we fix the canonical ``logarithm'' isomorphism
 $\lambda:G\rightarrow k_1G$, $\lambda(ab)=\lambda(a)+\lambda(b)$. Observe that $\lambda(1)$ is the zero of $k_1G$ and $k_1G$ has exponent 2, i.e, $\lambda(a)=-\lambda(a)$ for all $a\in G$. In the sequel, we define $k_*G$ by the quotient of the $\mathbb F_2$-graded algebra
 $$(\mathbb F_2,k_1G,k_1G\otimes_{\mathbb F_2} k_1G,k_1G\otimes_{\mathbb F_2} k_1G\otimes_{\mathbb F_2} k_1G,...)$$
 by the (graded) ideal generated by $\{\lambda(a)\otimes\lambda(ab),\,a\in D_G(1,b)\}$. In other words, for each $n\ge2$, 
$$k_nG:=T^n(k_1G)/Q^n(G),$$
where
$$T^n(k_1G):=k_1G\otimes_{\mathbb F_2} k_1G\otimes_{\mathbb F_2}...\otimes_{\mathbb F_2} k_1G$$ 
and $Q^n(G)$ is the subgroup generated by all expressions of type $\lambda(a_1)\otimes\lambda(a_2)\otimes...\otimes\lambda(a_n)$ such that for 
some $i$ with $1\le i< n$, there exist $b\in G$ such that $a_i\in D_G(1,b)$ and $a_i=a_{i+1}b$, which in symbols, means
\begin{align*}
 Q^n(G)&:=\langle\{\lambda(a_1)\otimes\lambda(a_2)\otimes...\otimes\lambda(a_n):\mbox{ exists }1\le i< n\mbox{ and }\,b\in G \\
 &\mbox{such that }a_i=a_{i+1}b\mbox{ and }a_i\in D_G(1,b)\}\rangle.
\end{align*}
\end{defn}
 
\begin{defn}[\cite{dickmann1998quadratic}, \cite{dickmann2006algebraic}]\label{2.4kt} 
$ $ Let $G$ be a formally real special group.
 \begin{enumerate}[a -]
  \item It holds [MC($G$)]  (i.e., $G$ satisfies "Marshall's conjecture") if for all $n\ge1$ and all forms $\varphi$ over $G$,
  $$\mbox{For all }\sigma\in X_G,\mbox{ if }\sigma(\varphi)\equiv0\,\mbox{mod }2^n\mbox{ then } \varphi\in I^nG.$$
  \item It holds [WMC($G$)] (i.e., $G$ satisfies "Weak Marshall's conjecture")  if for all $n\ge1$,  the multiplication by $\langle 1, 1 \rangle$ is an injection of $I^{n}(G)/I^{n+1}(G)$ into $I^{n+1}(G)/I^{n+2}(G)$.
  \item  It holds [SMC($G$)]  (i.e., $G$  satisfies "Strong Marshall's conjecture")  if for all $n\ge1$,  the multiplication by $\lambda(-1)$ is an injection of $k_n(G)$ into $k_{n+1}(G)$.
 \end{enumerate}
\end{defn}

It follows from  Proposition 4.6.(e) in \cite{dickmann1998quadratic} that  [MC($G$)] $\Rightarrow$ [WMC($G$)]; in Proposition 4.4 in  \cite{dickmann2003lam} is established  [SMC($G$)] $\Rightarrow$ [MC($G$)], for all reduced special group $G$. Now we apply Theorem APH to obtain the following:

\begin{prop} \label{MC}
 Let $G$ be a formally real special group.  Then $G$ satisfy Marshall [MC] (i.e., Marshall's signature conjecture holds in $G$)  iff $G$ satisfy [WMC] (i.e., Weak Marshall's conjecture holds in $G$).
\end{prop}
\begin{proof}
In the theorem 5.3 of \cite{dickmann1998quadratic} is established the equivalence of [MC] and [WMC] for all formally real special groups $G$  such that $2^k = \langle 1, 1 \rangle^k \notin I^{k+1}(G)$,for all $k\ge1$. But, it follows from Theorem APH that {\em all} formally real special groups automatically satisfies that property: otherwise $\langle 1, 1 \rangle^k$  will be hyperbolic and
thus $-1 \in Sat(G) = \bigcup_{k \in \mathbb{N}} D_G(2^k)$, contradicting that $G$ is a formally real special group.
\end{proof}


\begin{defn}[Definition 9.7 in \cite{dickmann2000special}]\label{igr1}
 An \textbf{inductive graded ring} (or \textbf{Igr} for short) is a structure $\mathcal{R}=((R_n)_{n\ge0},(h_n)_{n\ge0},\ast_{nm})$ where
\begin{enumerate}[i -]
    \item $R_0\cong\mathbb F_2$.
    \item $R_n$ is a group of exponent 2 with a distinguished element $\top_n$.
    \item $h_n:R_n\rightarrow R_{n+1}$ is a group homomorphism such that $h_n(\top_n)=\top_{n+1}$.
    \item For all $n\ge0$, $h_n=\ast_{1n}(\top_1,\_)$.
    \item The ring
    $$R=\bigoplus_{n\ge0}R_n$$
    is a commutative graded ring.
    \item For $0\le s\le t$ define
    $$h^t_s=\begin{cases}Id_{R_s}\mbox{ if }s=t\\
    h_{t-1}\circ...\circ h_{s+1}\circ h_s\mbox{ if }s<t.\end{cases}$$
    Then if $p\ge n$ and $q\ge m$, for all $x\in R_n$ and $y\in R_m$,
    $$h^p_n(x)\ast h^q_m(y)=h^{p+q}_{n+m}(x\ast y).$$
\end{enumerate}
A \textbf{morphism} between Igr's $\mathcal{R}$ and $\mathcal{R}'$ is a 
morphism of pointed groups and 
$$f=\bigoplus\limits_{n\ge0}f_n:R\rightarrow R'$$
is a morphism of commutative rings with unity (thus $\alpha_{n+1}  \circ h_n = h'_{n+1} \circ \alpha_n$). The category of inductive graded rings (in first version) and their morphisms will be denoted by $IGR$.
\end{defn}

In \cite{roberto2021ktheory} are considered some full subcategories of $IGR$ and  \cite{roberto2022graded} deals with limits and colimits of $IGR$ and these subcategories. 
A particularly useful sucategories is $IGR_h$, the full subcategory of Igr's $\mathcal{R}$ where for each $a \in R_1$, $\top_1 \ast_{1,1} a = a \ast_{1,1} a \in R_2$. Proposition 4.18 and Definition 4.19 therein  describes a functor $\Gamma : IGR_h \to pSG$ (the category of pre-special groups, that is equivalent to the category of pre-special hyperfields).

 Now, let $R\in IGR_h$. We have a pre-special group $\Gamma(\mathcal{R}) = (G(\mathcal{R}),+,-.\cdot,0,1)$ by the following: firstly, fix an isomorphism $e_R:(R_1,+_1,0_1,\top_1)\rightarrow (G(\mathcal{R}),\cdot,1,-1)$. This isomorphism makes, for example, an element $a\ast_{11}(\top_1+b)\in R_2$, $a,b\in R_1$ take the form $(e^{-1}_R(x))\ast_{11}(e^{-1}_R((-1)\cdot y))\in R_2$, $x,y\in G(\mathcal{R})$. 
 
 Now, let $\Gamma(\mathcal{R}):=G(\mathcal{R})$ and for $a,b, c, d \in R_1$ we have \ $\langle e_R(a) , e_R(b) \rangle \equiv \langle e_R(c) , e_R(d) \rangle$
 
 iff \ $a+b = c+d \in R_1$ \ and \ $a *_{11} b = c *_{11} d \in R_2$

If $\alpha = (\alpha_n)_{n \in \mathbb{N}} : \mathcal{R} \to \mathcal{R}'$ is a IGR-morphism, then $\Gamma(\alpha) : G(\mathcal{R}) \to G(\mathcal{R}')$ is the unique function (that turns out to be a pSG-morphism) such that $\Gamma(\alpha) = e _{R'} \circ \alpha_1 \circ e^{-1}_R$.

For each $G \in pSG$, the Igr's $W_*(G)$ and $k_*(G)$ belongs to the subcategory $IGR_h$ (Lemma 3.2 in \cite{dickmann1998quadratic}, \cite{dickmann2006algebraic} and Lemma 9.12 in \cite{dickmann2000special}) is defined a IGR-morphism $s_G : k_*(G) \to W_*(G)$ such that: $(s_G)_n (\lambda(a_1)\otimes \cdots \otimes \lambda(a_n)) = \langle 1, -a_1\rangle \otimes \cdots \otimes \langle 1, -a_n\rangle$ (Theorem 4.1 in \cite{dickmann2003lam}). In general  $(s_G)_n : k_n(G) \to I^n(G)/I^{n+1}(G)$ is a surjective homomorphism of pointed 2-groups and if $n =0,1,2$, then $(s_G)_n$ is an isomorphism of pointed 2-groups.

 Theorem 4.20 in \cite{roberto2021ktheory} establishes that the functor $k_*  : pSG \to IGR_q$ is left adjoint to the functor $\Gamma$ and the natural transformation that is the unity of this adjunction, $\kappa = (\kappa_G)_{G \in pSG}$, is such that for each $G \in pSG$, $\kappa_G :  G \to \Gamma(k_*(G))$, $g \mapsto \lambda(g)$ is a pSG-morphism that is an isomorphism of the underlying pointed 2-groups.
 
 In \cite{marshall1980abstract} M. Marshall proved that $\omega_G: G \to I(G)/I^2(G) $ $g \mapsto \langle 1, -g \rangle + I^2(G)$ is an isomorphism of pointed groups such that for each $a,b,c, d \in G$:
 $$\langle a , b \rangle \equiv_G \langle c , d \rangle\Rightarrow\langle 1 , -a \rangle \otimes\langle 1 , -b \rangle + I^3(G) = \langle 1 , -c \rangle \otimes\langle 1 , -d \rangle + I^3(G).$$ 

Thus, for each special group $G$, we have the following commutative diagram of pre-special groups and pSG-morphisms

$$(G \overset{\omega_G}\to \Gamma(W_*(G))) = (G \overset{\kappa_G}\to \Gamma(k_*(G)) \overset{\Gamma(s_G)}\to \Gamma(W_*(G)))$$

that is, moreover, natural in $G$. Now we are in position to state the:

\begin{prop} \label{k-stable} Let $G$ be a special group. Then 
 $ $
 \begin{enumerate}[i -]
  \item  $\Gamma(s_G):\Gamma(k_*(G))\rightarrow\Gamma(W_*(G)))$ is a pSG-isomorphism.
   \item  $\omega_G:G\rightarrow\Gamma( W_*(G))$ is a pSG-isomorphism.
  \item  $\kappa_G:G\rightarrow\Gamma(k_*(G))$ is a pSG-isomorphism.
 \end{enumerate}
  In particular, $\Gamma(k_*(G))$ and $\Gamma(W_*(G))$ are special groups. 
\end{prop}

\begin{proof}
This is essentially contained in the {\em proof of} Lemma 3.5 in \cite{dickmann2006algebraic}, but for convince the reader we provide some details:

First observe that, from axiom [SG4] is enough to show that for $x, y \in G$ are equivalent:\\
(1) $x \in D_G(1, y)$;\\ 
(2) $\lambda(x) \in D_{\Gamma(k_*(G))}(\lambda(1), \lambda(y)) $;\\
(3) $\langle 1, -x\rangle + I^2(G) \in D_{\Gamma(W_*(G))}(\langle 1, -1\rangle +I^2(G), \langle 1, -y\rangle + I^2(G))  $

 $(1) \Rightarrow (2)$: is clear from the definition of $k_*(G)$, just note that condition (2) is equivalent to $\lambda(x)\lambda(xy) = 0 \in  k_2(G)$
 
 $(2) \Rightarrow (3)$: this follows directly from $s_G : k_*(G) \to W_*(G)$ be a $IGR_h$-morphism 

 $(3) \Rightarrow (1)$: note that condition (3) is equivalent to 
$\langle 1, -x\rangle \otimes \langle 1, -xy\rangle + I^3(G)) = 0+ I^3(G))$. This means that $\langle 1, -x\rangle \otimes \langle 1, -xy\rangle \in I^3(G)$. Since $\dim(\langle 1, -x\rangle \otimes \langle 1, -xy\rangle) = 4 < 8 = 2^3$, then by Theorem APH, $\langle 1, -x\rangle \otimes \langle 1, -xy\rangle$ is an isotropic Pfister form. Thus it is an hyperbolic form, and then, by Proposition 2.2.(k) in \cite{dickmann2000special}, $x \in D_G(1,y)$.

\end{proof}

The notion of k-stable hyperbolic hyperfield $F$, i.e.  those such that canonical morphism $\kappa_{F} : F \to \Gamma(k_*(F)) \dot{\cup} \{0\}$ is an isomorphism of hyperfields, it is fundamental in \cite{roberto2022galois}. Thus the previous result establishes the:  

\begin{cor} \label{k-stable-co} Every special hyperfield $F$ is k-stable.
\end{cor}

The following result shows that the k-theory construction provides a very good encoding of the -- neither complete neither cocomplete -- category of special groups into the complete and cocomplete category of inductive graded rings.

\begin{prop} \label{epiK} The functor $k_* : SG \to IGR$ is full and faithful.
\end{prop}
\begin{proof}
We have to show that for each special groups $G_0$ and $G_1$ and any $\beta : k_*G_0\rightarrow k_*G_1$ be an {\em inductive graded} ring morphism between the associated  inductive graded rings of $k$-theory,  then there exist unique  SG-morphism $f : G_0\rightarrow G_1$ such that $\beta = k_*(f)$.

Proposition 3.6 in \cite{dickmann2006algebraic} establishes (from Lemma 3.5) that:
 for each special groups $G_0$ and $G_1$ and any $\beta : k_*(G_0)\rightarrow k_*(G_1)$ be a  {\em graded} ring morphism between the induced $k$-theory graded rings, such that $\beta_0 = id_{\mathbb{F}_2}$ and $G_1$ is a AP(3) special group,  then there exist a qSG-morphism $f : G\rightarrow H$ such that $\beta = k_*(f)$. Moreover, this $f$ is uniquely determined since $\beta_1 \circ \lambda_{G_0} = \lambda_{G_1} \circ f$ and $\lambda_{G_i} : G_i \to k_1(G_i)$ is an isomorphism of groups of exponent 2 that  preserves the distinguished elements ($-1_{G_i} \mapsto \lambda_{G_i}(-1_{G_i})$).

Thus the result follows since any special group $G_1$ satisfies is AP(3) (by Theorem APH), and since $\beta$ is a IGR-morphism then automatically  $\beta_0 = id_{\mathbb{F}_2}$ and $\beta_1 = k_1(f)$ implies that $f(-1_{G_0}) = -1_{G_1}$, thus $f$ the qSG-morphism $f$ is a SG-morphism.
\end{proof}

\begin{rem}
The previous  result can be derived, alternatively from Corollary \ref{k-stable-co} and Theorem 4.20 in \cite{roberto2021ktheory}: from an well known result on adjoint functors, a left adjoint is a full and faithful functor iff  the unity of the adjunction is an isomorphism. Thus $k_* : SG \to IGR_h$ is a full and faithful functor and, since $IGR_h \subseteq IGR$ is a full subcategory, then $k_* : SG \to IGR$ is full and faithful.
\end{rem}

\bibliographystyle{plain}
\bibliography{one_for_all.bib}

\end{document}